\documentclass[12pt,a4paper]{article}
\usepackage{amsfonts}

\usepackage{mathrsfs}

\usepackage{stmaryrd}
\usepackage{amsmath}
\usepackage{amssymb}

\usepackage{bbm}
\usepackage{mathrsfs}
\usepackage{graphicx}

\usepackage{enumerate}
\usepackage{theorem}
\usepackage{indentfirst}
\makeatletter
\def\tank#1{\protected@xdef\@thanks{\@thanks
        \protect\footnotetext[0]{#1}}}
\def\bigfoot{

    \@footnotetext}
\makeatother

\newcommand{\ea}{\end{array}}
\newtheorem{theorem}{Theorem}[section]

\newtheorem{lemma}{Lemma}[section]
\newtheorem{definition}{Definition}[section]
\newtheorem{Rem}{Remark}[section]

{\theorembodyfont{\rmfamily}
}

\newenvironment{proof}{Proof.}

\def \eref#1{\hbox{(\ref{#1})}}

\def\dt{d t}

\begin{document}
\title{{\Large \bf Stochastic 3D
Leray-$\alpha$ Model with Fractional Dissipation}
\footnote{Supported in part by NSFC (No. 11771187, 11571147),  NSF
of Jiangsu Province
 (No. BK20160004), the Qing Lan Project and
PAPD of Jiangsu Higher Education Institutions.}}

\author{{Shihu Li$^a$}, {Wei Liu$^b$}\footnote{Corresponding author: weiliu@jsnu.edu.cn},  {Yingchao Xie$^b$}
\\
 \small  $a.$  School of Mathematical Sciences, Nankai University, Tianjin 300071, China  \\
 \small  $b.$  School of Mathematics and Statistics, Jiangsu Normal University, Xuzhou 221116, China}
\date{}
\maketitle

\begin{center}
\begin{minipage}{145mm}
{\bf Abstract.}   In this paper, we establish the global well-posedness of stochastic 3D Leray-$\alpha$ model with general
fractional dissipation driven by multiplicative noise. This model is the stochastic 3D Navier-Stokes equations
 regularized through a smoothing kernel of order $\theta_1$ in the nonlinear term and a $\theta_2$-fractional Laplacian.
In the case of $\theta_1 \ge 0$ and $\theta_2 > 0$ with $\theta_1+\theta_2 \geq\frac{5}{4}$, we prove the global existence and uniqueness of the strong
solutions. The main results cover many existing works in the
deterministic cases, and also generalize some known results of
stochastic models such as stochastic hyperviscous Navier-Stokes
equations and classical stochastic 3D Leray-$\alpha$ model as our
special cases.


\vspace{3mm} {\bf Keywords:} Stochastic Leray-$\alpha$ model;
fractional Laplacian; Stochastic Navier-Stokes equations;
Well-posedness


\end{minipage}
\end{center}

\section{Introduction}
\setcounter{equation}{0}
 \setcounter{definition}{0}

The 3D Leray-$\alpha$ model of turbulence as a regularization of 3D
Navier-Stokes equations was first introduced by Leray \cite{L1} in
order to prove the existence of solutions to the Navier-Stokes
equation in $\mathbb{R}^3$. It has been studied in the following
form (cf. \cite{CTV,CHOT,Liu1})
\begin{eqnarray}
\left\{
  \begin{aligned}
 &d{u}-[\nu\Delta{u}-(v\cdot{\nabla})u-\nabla{p}]dt=f(u)dt, \\
    &u=v-\alpha^2\Delta{v}, \\
    &\nabla\cdot{u}=0, ~\nabla\cdot{v}=0,
  \end{aligned}
\right.
\end{eqnarray}
where $u$ (and $v$) are unknown fields, $\nu>0$ is the viscosity constant,
 $\alpha>0$  is a length-scale constant, $p$ denotes the pressure and
$f$ is an external force field acting on the fluid.

If $\alpha$ approaches to zero, then  (1.1) is reduced to the
classical 3D Navier-Stokes equations for incompressible fluids. It
is well-known that the uniqueness of global solutions of 3D
Navier-Stokes equations is among the most challenging problems of
contemporary mathematics. Many different types of modifications for
3D Navier-Stokes equations have been investigated  in the
literatures (see e.g. \cite{BM,RZ,T} and the references therein).
One would also like to discover whether there exists a noise
perturbation such that the uniqueness (pathwise or in law) holds for
the stochastic Navier-Stokes equations (cf. \cite{F1}). However, the
problem remains unsolved, in spite of considerable efforts.

In recent years, the fractional power of Laplacian has been paid a
lot of attentions (see e.g. \cite{BF,CSS,CV,CGV,Debbi,LR4,RZZ2,W,YZ,Z}
and the references therein). One can regularize the fluid equation
for the velocity by putting a fractional power of the Laplace
operator. One interesting point is to check for the limiting case
(critical case) whether one still have the global existence and
uniqueness of solutions (see e.g. \cite{BF,CV,W,YZ}) or not? The 3D
Leray-$\alpha$ model with fractional dissipation has been
extensively studied in the following form
\begin{eqnarray}\label{eq}
\left\{
  \begin{aligned}
   & d{u}+[\nu(-\Delta)^{\theta_2}u+(v\cdot{\nabla})u+\nabla{p}]dt=f(u)dt, \\
& u=v+\alpha^{2\theta_1}(-\Delta)^{\theta_1}v ,\\
 &\nabla\cdot{u}=0, ~\nabla\cdot{v}=0, \\
  \end{aligned}
\right. \end{eqnarray} where $(-\Delta)^{\theta_i}$ ($i=1,2$) are
 fractional Laplace operators, the parameter $\theta_1\geq0$ affects the
strength of the non-linear term and $\theta_2\geq0$ represents the
degree of viscous dissipation. The idea of regularizing the equation
by two terms was introduced by Olson and Titi in \cite{OT} for 3D
Lagrangian averaged Navier-Stokes-$\alpha$ model (LANS-$\alpha$
model).
The authors proposed  the idea in \cite{OT} that a weaker
nonlinearity and a stronger viscous dissipation could work together
to yield the well-posedness of the system. In particular, when
$\theta_1=0$, the model \eref{eq} becomes the hyperviscous
Navier-Stokes equations and it is well-known that this system has a
unique global solution for $\theta_2\geq\frac{5}{4}$ (see e.g.
\cite{L2,T,W}). When $\theta_2=1$, Ali  in \cite{A} studied the
global well-posedness of the critical Leray-$\alpha$ model
($\theta_1=\frac{1}{4}$) and also considered the convergence to a
suitable solution of Navier-Stokes equations. Note that the
uniqueness of global weak solutions to the model \eref{eq} for
$\theta_2=1$ and $\theta_1<\frac{1}{4}$ remains open with $L^2$
initial data. Moreover, the Leray-$\alpha$ models with more general
dissipation terms have been studied in \cite{BMR,P,Y}.

In the present paper, we will study the following stochastic 3D
Leray-$\alpha$ model with fractional dissipation on the 3D torus
$\mathbb{T}^3=[0,2\pi]^3$ with periodic boundary conditions:
\begin{eqnarray}
\left\{
 \begin{aligned}
    &d{u}+\left[\nu(-\Delta)^{\theta_2}u+(v\cdot{\nabla})u+\nabla{p}\right]dt=g(u)dW(t) ,   \\
 &u=v+\alpha^{2\theta_1}(-\Delta)^{\theta_1}v ,\\
 &\nabla\cdot{u}=0, ~\nabla\cdot{v}=0,\label{1.3}
  \end{aligned}
\right. \end{eqnarray}
where ${W(t)}$ is a cylindrical Wiener process in a separable Hilbert space. In order to emphasize the
stochastic effects and for the simplicity of exposition, we do not include a deterministic force $f$ in \eref{1.3},
but it's easy to show that all results of this paper could be easily extended to this more general case.

To the best of our knowledge, except for some special cases, there
is no result concerning stochastic 3D Leray-$\alpha$ model with
general fractional dissipation. In the case of $\theta_1=0$, the
stochastic fractional (or hyperviscous) Navier-Stokes equations have
been intensively studied  (see e.g. \cite{Debbi,F,RZZ1,RZZ3,WX} and
references within). Chueshov and Millet in \cite{CM} proved the
well-posedness and large deviation principle of stochastic 3D
Leray-$\alpha$ model in the case of $\theta_1=\theta_2=1$ (see also
\cite{DS}). The well-posedness and irreducibility of 3D
Leray-$\alpha$ model driven by L\'{e}vy noise have been studied in
\cite{BHR,FHR}. The $\alpha$-approximation of stochastic
Leray-$\alpha$ model to the stochastic Navier--Stokes equations was
established in \cite{BR,CK}. In addition,  when the viscosity
constant $\nu=0$, Barbato, Bessaih and Ferrario in \cite{BBF}
studied the 3D stochastic inviscid Leray-$\alpha$ model. The purpose
of this paper is to investigate the global existence and uniqueness
of solutions to the equation (\ref{1.3}) under certain assumptions
on the parameters $\theta_1$ and $\theta_2$. More precisely, we
prove that if $\theta_1\geq0$  and $\theta_2>0$ satisfy
$\theta_1+\theta_2 \geq \frac{5}{4}$, the equation \eref{1.3} has a
unique (probabilistically) strong solution (see Theorems 3.1 and
4.1).

Comparing with the results in the deterministic case,
Yamazaki in \cite{Y} obtained an unique global solution to the 3D
Leray-$\alpha$ model with the logarithmical dissipation when
$\theta_2\geq\frac{1}{2}$ and $\theta_1+\frac{\theta_2}{2} \geq \frac{5}{4}$;  then it was improved by Pennington in
\cite{P} for $\theta_2>\frac{1}{2}$ and $\theta_1+\theta_2\geq
\frac{5}{4}$. Moreover, Barbato, Morandin and Romito in \cite{BMR}
proved the 3D Leray-$\alpha$ model with the logarithmical
dissipation is well-posed for $\theta_1+\theta_2 \geq \frac{5}{4}$
with $\theta_1$ and $\theta_2\geq0$, this result mainly focused on the
optimal condition on the correction to the dissipation and was
obtained by analyzing energy dispersion over dyadic shells, which is
a completely different approach with the one employed in this
work. Another related work is about 3D regularized Boussinesq
equations in \cite{BF}, but their results are  restricted to the
case of $\frac{1}{2}<\theta_2<\frac{5}{4}$ and $\theta_1+\theta_2
=\frac{5}{4}$.

The proof of our main results is divided into two cases. If the
viscous dissipation is strong enough, $i.e.$ for
$\theta_2>\frac{1}{2}$, we can apply the generalized  variational
approach (cf. \cite{LR1,LR3}) to get the well-posedness with initial
data in $\mathbb{H}^0$. The reason is that $(-\Delta)^{\theta_2}$ is
smoothing by $2\theta_2$ derivatives, while the $(v\cdot\nabla) u$
has one derivative. Let $\mathbb{H}^s$ denote the Sobolev space of
divergence free vector fields (see \eref{1.22}). We shall consider
the Geland triple $\mathbb{H}^{\theta_2}\subset \mathbb{H}^0 \subset
\mathbb{H}^{-\theta_{2}},$ and get the existence and uniqueness of
strong solutions when $\theta_1\ge 0$ and $\theta_2 > \frac{1}{2}$
with $\theta_1+\theta_2 \geq \frac{5}{4}$. In this case, we can
check the coefficients satisfy the local monotonicity and coercivity
conditions similarly as the second named author and R\"{o}ckner did
for various types of SPDEs in \cite{LR1} (see also
\cite{Liu2,LR2,LR4}). It is known that the 3D Navier-Stokes
equations ($\theta_2=1,~\theta_1=0$) lie outside the framework in
\cite{LR3} (see Example 5.2.23 there), hence our result illustrates
that if we regularize the 3D Navier-Stokes equations by putting a
fractional power of Laplacian for $\theta_2\geq\frac{5}{4}$, then
the corresponding stochastic hyperviscous Navier-Stokes equations
(i.e. $\theta_2\geq\frac{5}{4}, ~\theta_1=0$) are also included in
the generalized variational framework.

On the other hand, the case $0<\theta_2\leq\frac{1}{2}$ is much more
difficult to handle. Since the dissipation term is not strong enough
to control the non-linear term  $(v\cdot\nabla) u$, the uniqueness
of solutions with initial data in $\mathbb{H}^0$ seems unavailable.
To get the well-posedness for \eref{1.3}, we work in the phase space
$\mathbb{H}^1$ with the initial data in $\mathbb{H}^1$. And we
get the global existence of unique strong solution which is also
continuous with respect to $\mathbb{H}^1$ norm  for the case
$\theta_1\ge 0$ and $\theta_2 > 0$ with $\theta_1+\theta_2\geq\frac{5}{4}$.
Note that the variational approach is not applicable to \eref{1.3}
in this case, so we use a different approach to get the well-posedness.
Firstly, we get the existence of a unique local strong solution when
$\theta_1\ge 0$ and $\theta_2>0$ satisfying $\theta_1+\theta_2
>\frac{3}{4}$ based on the work of \cite{RZZ1}.  Then we prove that the solution is global for
$\theta_1\ge 0$ and $\theta_2> 0$ with $\theta_1+\theta_2 \geq
\frac{5}{4}$. In \cite{GV,RZZ1} the authors show that, when the
noise is linear multiplicative, the local solution is global with a
high probability if the initial data is sufficiently small, or if
the noise coefficient is sufficiently large. Different with the
arguments in \cite{GV,RZZ1}, we consider the general multiplicative
noise including but not limit to the linear case, and show that the
local solution is global almost surely. The proof of global
existence of solution is based on a stochastic version Gronwall's
lemma from \cite{GZ1} and some stopping time techniques. But it is
clear that the proof of the main results is more involved than the
arguments in \cite{GZ1,RZZ1}, as we need to maintain the balance of
the mixed fractional dissipation terms in \eref{1.3}. Another main
difficulty arising is to prove appropriate $\mathbb{H}^s$-norm
estimate with $s>0$. Unlike the case $s=0$, we do not have the
cancellation property
$\int\Lambda^s(v\cdot\nabla{u})\Lambda^s{u}dx=0$ anymore. Inspired
by the works \cite{BF,CGV, RZZ1,RZZ2}, we frequently use the
commutator estimate to show that the $\mathbb{H}^1$-norm can be
controlled. Due to the lack of the cancellation property, we are not
able to show that the solution is in
$L^p(\Omega;L^\infty_{loc}([0,\infty);\mathbb{H}^1))\cap{L}^2(\Omega;L^2_{loc}([0,\infty);\mathbb{H}^{\theta_2+1
}))$ for $\theta_1\ge 0$ and $\theta_2> 0$ with $\theta_1+\theta_2
\geq \frac{5}{4}$. However, we prove that the solution is in
$L^p(\Omega;L^\infty_{loc}([0,\infty);\mathbb{H}^1))\cap{L}^2(\Omega;L^2_{loc}([0,\infty);\mathbb{H}^{\theta_2+1
}))$ in the subcritical case (i.e. $\theta_1\geq 0, \theta_2> 0$,
$\theta_1+\theta_2 > \frac{5}{4}$).

The main results of this paper illustrate how the non-linearity and
viscous dissipation can be balanced with each other to yield the well-posedness of
 stochastic 3D Leray-$\alpha$ model. In particular, if we take $\theta_1= 0$ and $\theta_2=1$ in \eref{1.3}, the corresponding
stochastic 3D Navier-Stokes equations lie in the region of the local
well-posedness and outside the region of the global well-posedness
of our main results. Moreover, our global well-posedness results are
applicable for the case
 $\theta_2=1$,
$\theta_1=\frac{1}{4}$ and the case $\theta_1=0$,
$\theta_2\geq\frac{5}{4}$, which are corresponding to the the
stochastic critical Leray-$\alpha$ model (see \cite{A} for
deterministic case) and hyperviscous Navier-Stokes equations. Hence
our results cover and generalize some corresponding results in
\cite{A,Debbi,DS,F,GZ1,RZZ1}. And we believe that the methods
presented in this paper are also useful for tackling other types of
SPDEs with fractional Laplacian.

We should mention that there also exist many works concerning other
types of SPDE with fractional Laplacian such as stochastic
fractional Burgers equation, stochastic quasi-geostrophic equation,
stochastic fractional Euler equations, stochastic fractional
reaction-diffusion equation and stochastic fractional Boussinesq
equations (see e.g. \cite{BD,CGS,CGV,Debbi,GZ2,HSL,LR4,RZZ1,RZZ2}
and the references therein).

 The paper
is organized as follows. In Section 2 we introduce some notations
and preliminaries. In Section 3 we prove the well-posedness for
equation \eref{1.3} with initial data in $\mathbb{H}^0$ when
$\theta_1\ge 0,\theta_2
> \frac{1}{2}$ and $\theta_1+\theta_2 \geq
\frac{5}{4}$. In Section 4, we establish the well-posedness for
equation \eref{1.3} with initial data in $\mathbb{H}^1$ when
$\theta_1\ge 0, \theta_2>0$ and $\theta_1+\theta_2 \geq
\frac{5}{4}$. At the end of this Section, we also show that there
exist the finite moments of $\mathbb{H}^1$ norm of the solution at
any given deterministic time $t$ in the subcritical case (i.e.
$\theta_1\geq 0, \theta_2> 0$, $\theta_1+\theta_2 > \frac{5}{4}$).

\section{Preliminaries}
\setcounter{equation}{0}
 \setcounter{definition}{0}

In this section, we introduce some notations and  preliminaries which are commonly used in the analysis of fluid equations.

We denote by $L^p=L^p(\mathbb{T}^3)^3$  the usual Lebesgue space
over $\mathbb{T}^3$ with the norm $\|\cdot\|_{L^p}$. As usual in the
periodic setting, we can restrict ourself to deal with initial data
with vanishing spatial average; then the solutions will enjoy the
same property at any fixed time $t>0$.

Since we work with periodic boundary condition, we can expand the
velocity in Fourier series as
$$u(x)=\sum_{k\in\mathbb{Z}_0^3}\hat{u}_ke^{ik\cdot{x}},~~\mbox{with}~\hat{u}_k\in\mathbb{C}^3,~\hat{u}_{-k}=\hat{u}^*_k~~\text{for every}~k,$$
 where
$\mathbb{Z}^3_0=\mathbb{Z}^3\setminus\{0\}$ and $\hat{u}^*_k$ denotes the complex conjugate of $\hat{u}_k$.

For $s\in\mathbb{R}$,
the Sobolev spaces $H^s(\mathbb{T}^3)^3$ can be represented as
$$H^s=\left\{u(x)=\sum_{k\in\mathbb{Z}_0^3}\hat{u}_ke^{ik\cdot{x}}:\hat{u}_{-k}=\hat{u}^*_k,\|u\|_s^2<\infty\right\},$$
where
$$\|u\|_s^2=\sum_{k\in\mathbb{Z}_0^3}|k|^{2s}|\hat{u}_k|^2.$$
In the Fourier space, the divergence free condition can be formulated as
$$\hat{u}_k\cdot{k}=0~~\mbox{for every}~k.$$
Define the divergence free Sobolev space by
\begin{eqnarray}
\mathbb{H}^s:=\big\{ u\in{H}^s:\hat{u}_k\cdot{k}=0 ~~\mbox{for
every}~k \big\},\label{1.22}
\end{eqnarray} which is a Hilbert space with scalar product
$$\langle{u},v\rangle_{\mathbb{H}^s}=\sum_{k\in\mathbb{Z}_0^3}|k|^{2s}\hat{u}_k\cdot{\hat{v}}_{-k}.$$

Following the standard notation, we denote the norm in space
$\mathbb{H}^0$ by $\|u\|_{L^2}$ and inner product
$\langle{u},v\rangle=\sum_k \hat{u}_k\cdot\hat{v}_{-k} $. Thus,
$\mathbb{H}^0$ is the Hilbert space of  $L^2$-integrable functions
on $\mathbb{T}^3$ taking values in $\mathbb{R}^3$ which are
divergence free and have zero mean.  For simplicity, we also
identify the continuous dual space of $\mathbb{H}^s$ as
$\mathbb{H}^{-s}$ with the dual action of $\mathbb{H}^{-s}$ on
$\mathbb{H}^s$ by the same notation $\langle{u},v\rangle=\sum_k
\hat{u}_k\cdot\hat{v}_{-k}$. In particular, we denote the norm of
$\mathbb{H}^1$ by $\|u\|$.

 The nonlocal operator
$\Lambda^{s}$ is defined as
$$\Lambda^s u:=\sum_{k\in\mathbb{Z}_0^3}|k|^{s}\hat{u}_ke^{ik\cdot{x}},$$
hence $\Lambda^2=-\Delta$ . Note that $\Lambda^s$ maps $H^r$ onto
$H^{r-s}$ and
$$\|u\|_s^2=\sum_{k\in\mathbb{Z}_0^3}|k|^{2s}|\hat{u}_k|^2=\|\Lambda^s{u}\|_{L^2}^2.$$

Denote by $P_\sigma$ the Leray-Helmholtz projection from $H^\beta$
to $\mathbb{H}^\beta$. It's well-known that the operators $P_\sigma$
and $\Lambda^s$ are commutative. Note that the space
$H^{s+\epsilon}$ is compactly embedded in $H^{s}$ (resp.
$\mathbb{H}^{s+\epsilon}$ is compactly embedded in $\mathbb{H}^{s}$)
for any $\epsilon>0$, moreover we have the following Sobolev
embedding theorem (see e.g. \cite{BF,ES}).
\begin{lemma}If $0\leq{}s<\frac{3}{2}$ and
$\frac{1}{p}+\frac{s}{3}=\frac{1}{q}$, then $H^s\subset{L}^p$.
Moreover, there is a constant $C=C(s,p)>0$ such that
$$\|f\|_{L^p}\leq C\|\Lambda^s{f}\|_{L^q}.$$
If $s=\frac{3}{2}$, then for any finite $p$,
$$\|f\|_{L^p}\leq C\|f\|_{s}$$
and if $s>\frac{3}{2}$, then
$$\|f\|_{L^\infty}\leq C\|f\|_{s}.$$
\end{lemma}

\vspace{3mm}
Define the bilinear operator
$B:\mathbb{H}^1\times{\mathbb{H}^1}\rightarrow{\mathbb{H}}^{-1}$ by
$$\langle{B}(u,v),w\rangle:=\int_{\mathbb{T}^3}\left((u\cdot\nabla){v}\right)\cdot{w}dx,$$
i.e. $B(u,v)=P_\sigma\left((u\cdot\nabla){v}\right)$ for smooth
vectors $u$ and $v$.

We list some well-known properties of the bilinear operator $B$
below (see e.g. \cite{RT,BF}).
\begin{lemma}
For any $u,v,w\in{\mathbb{H}^1}$, we have
\begin{eqnarray}
\langle{B}(u,v),w\rangle=-\langle{B}(u,w),v\rangle,~~\langle{B}(u,v),v\rangle=0.\label{2.1}
\end{eqnarray}
\eref{2.1} holds more generally for any $u,v,w$ giving a meaning to
the trilinear forms, as stated precisely in the following:
\begin{eqnarray}
\langle{B}(u,v),w\rangle\leq{C}\|u\|_{m_1}\|v\|_{m_2+1}\|w\|_{m_3},\label{2.2}
\end{eqnarray}
with the nonnegative parameters fulfilling
$$m_1+m_2+m_3\geq\frac{3}{2}~~~~~\text{if}~m_i\neq\frac{3}{2}~\text{for~any~i}$$
or$$m_1+m_2+m_3>\frac{3}{2}~~~~~\text{if}~m_i=\frac{3}{2}~\text{for~some~i}.$$
\end{lemma}

Put $G=(I+\alpha^{2\theta_1}\Lambda^{2\theta_1})^{-1}$, then $G$ is
a linear operator. For $u=v+\alpha^{2\theta_1}\Lambda^{2\theta_1}v$,
we have $v=Gu$. Denote $B(u):=B(Gu,u)$.  By applying $P_\sigma$ to
Eq.~\eref{1.3} we remove the pressure term  and reformulate it as
the following abstract stochastic evolution equation:
\begin{eqnarray}
\left\{ \begin{aligned}
&d{u(t)}+\nu{\Lambda}^{2\theta_2}u(t)dt+B(u(t))dt=g(u(t))dW(t),\\
&u(0)=u_0,
\end{aligned} \right.\label{spde}
\end{eqnarray} where $W(t)$ is a cylindrical Wiener process
in a separable Hilbert space $U$ w.r.t. a complete filtered
probability space $(\Omega,\mathscr{F}, \{\mathscr{F}_t\}_{t\geq0}, \mathbb{P})$.

The regularization effect of the nonlocal operator involved in the
relation between $Gu$ and $u$ is described by the following lemma
(see \cite[Lemma 2.2]{A}).

\begin{lemma}Let $0\leq\beta\leq2\theta_1$, $s\in{\mathbb{R}}$ and $u\in{H^s}$,
then $Gu\in{H^{s+\beta}}$ and there exists a constant
$C=C_{\alpha,\beta}>0$ such that
$$\|Gu\|_{s+\beta}\leq{C}\|u\|_s.$$
\end{lemma}

Define the commutator
$$[\Lambda^s,f]g=\Lambda^s(fg)-f\Lambda^sg.$$
The following commutator estimate is very important for later use
(see \cite{BF,KP}).
\begin{lemma}(Commutator estimate) Suppose that $s>0$, $p,p_2,p_3\in(1,\infty)$ and $p_1,p_4\in(1,\infty]$ satisfy
\begin{eqnarray}
\frac{1}{p}\geq\frac{1}{p_1}+\frac{1}{p_2},~~~\frac{1}{p}\geq\frac{1}{p_3}+\frac{1}{p_4}.\label{2.5}
\end{eqnarray}Then we have
$$\|[\Lambda^s,f]g\|_{L^p}\leq{C}\left(\|\nabla{f}\|_{L^{p_1}}\|\Lambda^{s-1}{g}\|_{L^{p_2}}+\|\Lambda^s{f}\|_{L^{p_3}}\|g\|_{L^{p_4}}\right).$$
\end{lemma}

We recall the following important product estimate (see, e.g.
\cite{BF,RZZ2}).
\begin{lemma} Suppose that $s>0$,
$p,p_2,p_3\in(1,\infty)$, $p_1,p_4\in(1,\infty]$ satisfy \eref{2.5}.
We have
$$\|\Lambda^s(fg)\|_{L^p}\leq{C}\left(\|{f}\|_{L^{p_1}}\|\Lambda^{s}{g}\|_{L^{p_2}}+\|\Lambda^s{f}\|_{L^{p_3}}\|g\|_{L^{p_4}}\right).$$
\end{lemma}

For any Hilbert space $K$, we use
($L_2(U,K)$,$\|\cdot\|_{L_2(U,K)}$) to denote
 the space of all Hilbert-Schmidt
operators from $U$ to $K$. In this paper we use $C$ to denote some generic constant which may
change from line to line.

\section{Main results for $\theta_2>\frac{1}{2}$ with initial data in $\mathbb{H}^0$}
\setcounter{equation}{0}
 \setcounter{definition}{0}
In this section, we will prove the existence and uniqueness of
Eq.~\eref{spde} with initial data in $\mathbb{H}^0$ when
$\theta_2>\frac{1}{2}$. To this end, we first impose the following
assumptions on $g$.

\medskip
\noindent\textbf{{Hypothesis} (3.1)}  Suppose that  $g$ is measurable mapping from $\mathbb{H}^{\theta_2}$ to
${L_2(U,\mathbb{H}^0)}$ and satisfies the following conditions:
\begin{enumerate}[(i)]
  \item  There exists a constant $C>0$ such that, for any $u\in\mathbb{H}^{\theta_2}$,
$$\|g(u)\|_{L_2(U,\mathbb{H}^0)}^2\leq C(1+\|u\|_{L^2}^2).$$

  \item There exists a constant $C>0$ such that, for any $u,v \in\mathbb{H}^{\theta_2}$,
$$ \|g(u)-g(v)\|_{L_2(U,\mathbb{H}^0)}^2\leq \nu \|u-v\|_{\theta_2}^2
+ C(1+\rho(v))\|u-v\|_{L^2}^2,$$
where $\rho:\mathbb{H}^{\theta_2}\to[0,+\infty)$ is a
measurable and locally bounded function in $\mathbb{H}^{\theta_2}$
such that
$$\rho(v)\leq C(1+\|v\|_{\theta_2}^2)(1+\|v\|_{L^2}^2).$$
\end{enumerate}

 Next lemma plays an essential role and the proof can be found from \cite[Lemma 4.2]{OT}.

\begin{lemma}\label{N.R.}
 Suppose that  $\theta_1\ge 0$ and $\theta_2 > \frac{1}{2}$ with $\theta_1+\theta_2 \geq \frac{5}{4}$, then
$B:{\mathbb{H}}^{2\theta_1+\theta_2}\times\mathbb{H}^{\theta_2}\rightarrow\mathbb{H}^{-\theta_2}$
is well defined. Moreover, for $u\in{\mathbb{H}}^{2\theta_1+\theta_2}$ and $v,w\in{\mathbb{H}}^{\theta_2}$,
$$|\langle B(u,v),w\rangle|\leq C \left(\|u\|_{2\theta_1+\theta_2}\|w\|_{L^2}+\|u\|_{2\theta_1}\|w\|_{\theta_2}\right)\|v\|_{\theta_2}.$$
\end{lemma}

\begin{definition}
A continuous adapted $\mathbb{H}^0$-valued process
$\{u(t)\}_{t\in[0,T]}$ is called a (strong) solution of \eref{spde},
if for its $dt\otimes\mathbb{P}$-equivalent class $\bar{u}$, we have
$\bar{u} \in L^2([0,T]\times\Omega;\mathbb{H}^{\theta_2})$ and
 $\mathbb{P}$-$a.s.$,
$$u(t)+\nu\int_0^t\Lambda^{2\theta_2}\bar{u}(s)ds+\int_0^tB(\bar{u}(s))ds=u_0+\int_0^t g(\bar{u}(s))dW(s), \ \ t\in[0,T].$$
\end{definition}

The first main result of this paper is given in the next statement.

\begin{theorem}\label{Th1}
Suppose that $\theta_1\ge 0$, $\theta_2 > \frac{1}{2}$ with
$\theta_1+\theta_2 \geq \frac{5}{4}$ and the Hypothesis (3.1) hold.
We have that, for any $u_0~{\in}~L^p(\Omega;\mathbb{H}^0)$ with
$p\geq4$, \eref{spde} has a unique strong solution
$\{u(t)\}_{t\in[0,T]}$, which satisfies
$$\mathbb{E}\left(\sup_{t\in[0,T]}\|u(t)\|_{L^2}^p+\int_0^T\|u(t)\|_{\theta_2}^2dt\right)<\infty.$$
Moreover, the solution $\{u(t)\}_{t\in[0,T]}$ is a Markov process.
\end{theorem}

\begin{proof}
Now we consider the following Geland triple
$$\mathbb{H}^{\theta_2}\subset \mathbb{H}^0 \subset
\mathbb{H}^{-\theta_{2}}.$$ We first note that the following
mappings
$$\Lambda^{2\theta_2}:\mathbb{H}^{\theta_2}\rightarrow \mathbb{H}^{-\theta_2},~~B: \mathbb{H}^{\theta_2}\rightarrow
\mathbb{H}^{-\theta_2}$$ are well defined. In particular, by Lemma
2.2, we have
$$\langle B(Gv,u),w\rangle=-\langle B(Gv,w),u\rangle,~~\langle B(Gv,u),u\rangle=0,~~\text{for}~v,u,w\in \mathbb{H}^{\theta_2}.$$

Let $F:\mathbb{H}^{\theta_2}\rightarrow{\mathbb{H}}^{-\theta_{2}}$
be defined by
 $$F(u):=-\nu \Lambda^{2\theta_2}u-B(u),\ \ \text{for ~all}~u\in {\mathbb{H}}^{\theta_2}.$$
We only need to verify that all conditions of Theorem 5.1.3 in \cite{LR3} hold for \eref{spde}.

(1) Since $B$ is bilinear map,  the hemicontinuity
 of $F$ is obvious.

(2) Note that $\langle B(Gu,u),u \rangle =0$, it is easy to verify the
following coercivity condition:
$$ 2\langle F(u),u\rangle +\|g(u)\|_{L_2(U,\mathbb{H}^0)}^2 \leq -\nu \|u\|^2_{\theta_2}+C(1+\|u\|_{L^2}^2).$$

(3) By the Young's inequality, Lemma 2.3 and Lemma 3.1, for $u_1,u_2\in{\mathbb{H}}^{\theta_2}$ we have
\begin{eqnarray}
&&\langle
B(u_1)-B(u_2),u_1-u_2\rangle\nonumber\\=\!\!\!\!\!\!\!\!&&\langle
B(G(u_1-u_2),u_2),u_1-u_2\rangle
\nonumber\\\leq\!\!\!\!\!\!\!\!&&C\left(\|G(u_1-u_2)\|_{2\theta_1+\theta_2}\|u_1-u_2\|_{L^2}+\|G(u_1-u_2)\|_{2\theta_1}\|u_1-u_2\|_{\theta_2}\right)\|u_2\|_{\theta_2}
\nonumber\\\leq\!\!\!\!\!\!\!\!&&C\|u_1-u_2\|_{\theta_2}\|u_1-u_2\|_{L^2}\|u_2\|_{\theta_2}
\nonumber\\\leq\!\!\!\!\!\!\!\!&&\frac{\nu}{4}\|u_1-u_2\|^2_{\theta_2}+C_\nu\|u_2\|^2_{\theta_2}\|u_1-u_2\|_{L^2}^2.
\end{eqnarray}
Then by {Hypothesis} (3.1) we can deduce that
 \begin{eqnarray*}
&&2\langle F(u_1)-F(u_2),u_1-u_2\rangle +\|g(u_1)-g(u_2)\|_{L_2(U,\mathbb{H}^0)}^2\\
\leq\!\!\!\!\!\!\!\!&&C\left(1+\rho(u_2)+\|u_2\|^2_{\theta_2}\right)\|u_1-u_2\|_{L^2}^2,
\end{eqnarray*}
hence the local monotonicity condition holds.

(4) By Lemma 2.2 and Lemma 3.1, for any $u,w\in{\mathbb{H}}^{\theta_2}$
$$|\langle B(u),w\rangle|=|\langle B(Gu,w),u\rangle|\leq
C\left(\|Gu\|_{2\theta_1+\theta_2}\|u\|_{L^2}+\|Gu\|_{2\theta_1}\|u\|_{\theta_2}\right)\|w\|_{\theta_2},$$
from which and Lemma 2.3 we have the following growth condition:
$$\|F(u)\|^2_{-\theta_2}\leq C \|u\|^2_{\theta_2}( 1+\|u\|_{L^2}^2).$$

Therefore, all conclusions follow from Theorem 5.1.3 in \cite{LR3}.
The proof is completed. \hspace{\fill}$\Box$
\end{proof}

\section{Main results for $\theta_2>0$ with initial data in $\mathbb{H}^1$ }
\setcounter{equation}{0}
 \setcounter{definition}{0}

In this section, we show the existence and uniqueness of solutions
to Eq.~\eref{spde} for $\theta_2>0$ with initial data in
$\mathbb{H}^1$. Here we first prove that Eq.~\eref{spde} is local
well-posedness and then show that the local solution is global.
Firstly, we make the following assumptions of $g$ such that
Eq.~\eref{spde} has a unique local strong solution.

\medskip
\noindent\textbf{{Hypothesis} (4.1)} Suppose that $g$ is measurable
mapping from $\mathbb{H}^0$ to ${L_2(U,\mathbb{H}^0)}$ and it
satisfies the following conditions:
\begin{enumerate}[(i)]
\item  For all $s\in[1,2]$, $g$ is an operator from $\mathbb{H}^s$ to
$L_2(U,\mathbb{H}^s)$ and there exists a locally bounded function
$\rho_1$ on $\mathbb{R}$ such that for all $u\in\mathbb{H}^s$
$$\|\Lambda^sg(u)\|_{L_2(U,\mathbb{H}^0)}\leq
\rho_1(\|{u}\|)(1+\|u\|_{s}). $$

\item  There exist locally bounded functions $\rho_2$ and
$\rho_3$ on $\mathbb{R}$ such that for all $u,v\in\mathbb{H}^1$

$$\|g(u)-g(v)\|_{L_2(U,\mathbb{H}^0)}\leq \left(\rho_2(\|{u}\|)
+\rho_3(\|{v}\|)\right)\|u-v\|_{L^2}. $$

\end{enumerate}

In order to prove the global well-posedness, we also need to impose
some further assumptions on $g$.

\medskip
\noindent\textbf{{Hypothesis} (4.2)} There exists a constant $C$
such that

$$  \|g(u)\|_{L_2(U, \mathbb{H}^0)}^2\leq
C(1+\|u\|_{L^2}^2),\ \ \text{for ~all}\ u\in\mathbb{H}^{\theta_2}$$
and
$$\|\Lambda{g}(u)\|_{L_2(U,\mathbb{H}^0)}^2\leq
C(1+\|u\|^2),\ \ \text{for ~all}\  u\in\mathbb{H}^{\theta_2+1}.$$

\begin{Rem} The assumptions of $g$ in Hypotheses (4.1) and (4.2) may be shown to cover a wider class of examples,
 including but not limit to the classic cases of additive and linear multiplicative noise.

\end{Rem}

 We recall the following notions of local, maximal and global
solutions of Eq.~\eref{spde}.

\begin{definition}
Fix a stochastic basis $(\Omega, \mathscr{F}, \mathbb{P}, \mathscr
{F}_t, W)$.
\begin{enumerate}[(i)]
  \item A local strong solution of $\eref{spde}$ is a pair $(u,\tau)$, where
$\tau$ is an $\mathcal {F}_t$-stopping time and $(u(t))_{t\geq0}$ is
a predictable $\mathbb{H}^1$-valued process such that
$u(\cdot\wedge\tau)\in{L}^2(\Omega;L^2_{loc}([0,\infty);\mathbb{H}^{\theta_2+1}))$,
$$u(\cdot\wedge\tau)\in{C}([0,\infty);\mathbb{H}^{1})~~~\mathbb{P}{\text -a.s.},$$
and for every
$t\geq0,\zeta\in\cap_{l=1}^\infty\mathbb{H}^{l}$,~$\mathbb{P}$-a.s.,
\begin{eqnarray}
\langle u(t\wedge\tau),\zeta\rangle+\int_0^{t\wedge\tau}\langle
\nu\Lambda^{2\theta_2}u+ B(u),\zeta\rangle ds=\langle
u_0,\zeta\rangle+\int_0^{t\wedge\tau}\langle
{g}(u)dW,\zeta\rangle.\label{4.3}
\end{eqnarray}

  \item We say that local pathwise uniqueness holds if given any pair $(u^1,\tau^1)$ and $(u^2,\tau^2)$ of local strong solutions of
\eref{spde} with the same initial condition, the following holds:
$$\mathbb{P}\left\{u^1(t)=u^2(t);\forall t\in[0,\tau^1\wedge\tau^2]\right\}=1.$$

  \item A maximal strong solution of $\eref{spde}$ is a pair
$(u^R,\tau_R)_{R\in{\mathbb{N}}}$ such that for each
$R\in\mathbb{N}$, the pair $(u^R,\tau_R)$ is a local strong
solution, $\tau_R$ is increasing such that
$\xi:=\lim_{R\rightarrow\infty}\tau_R>0,~\mathbb{P}$-a.s. and
\begin{eqnarray}\sup_{t\in[0,\tau_R]}\|u^R(t)\|\geq R,~~~~\mathbb{P}{\text -a.s.} \text{on the
set}~\{\xi<\infty\}.\label{4.3c}
\end{eqnarray}

  \item If the local pathwise uniqueness holds, then $\xi$ does not
depend on the sequences
$(u^R)_{R\in{\mathbb{N}}}$,
$(\tau_R)_{R\in{\mathbb{N}}}$. In this case
we denote the maximal solution by
$(u,(\tau_R)_{R\in{\mathbb{N}}},\xi)$, and we say a maximal strong
solution $(u,(\tau_R)_{R\in{\mathbb{N}}},\xi)$ is global if
$\xi=\infty~\mathbb{P}$-a.s.
\end{enumerate}
\end{definition}

Now we formulate the main result in this section concerning the
global well-posedness of Eq.~\eref{spde}.

\begin{theorem}\label{Th2}
Suppose that $\theta_1\ge 0, \theta_2 > 0$ with $\theta_1+\theta_2
\geq\frac{5}{4}$ and the {Hypotheses} (4.1), (4.2) hold. Then for
any $u_0~{\in}~L^2(\Omega;\mathbb{H}^{1})$ \eref{spde} has a unique
global strong solution. More precisely, there exists a unique
predictable $\mathbb{H}^1$-valued process $\{u(t)\}_{t\geq0}$ such
that
$$u\in{L}^2_{loc}([0,\infty);\mathbb{H}^{\theta_2+1})\cap{C}([0,\infty);\mathbb{H}^{1})~~~\mathbb{P}{\text -a.s.},$$
and for every $t\geq0,\zeta\in\cap_{l=1}^\infty\mathbb{H}^{l}$
\begin{eqnarray}
\langle u(t),\zeta\rangle+\int_0^{t}\langle \nu\Lambda^{2\theta_2}u+
B(u),\zeta\rangle ds =\langle u_0,\zeta\rangle+\int_0^{t}\langle
{g}(u)dW,\zeta\rangle\ \ \mathbb{P}\text{-a.s}.\label{4.3b}
\end{eqnarray}
\end{theorem}

The proof of Theorem 4.1 will be given in the section 4.2.

\subsection{Local existence and uniqueness}
In this section, we establish the existence of local solution and
maximal solution for Eq.~\eref{spde} with initial data in
$L^2(\Omega;\mathbb{H}^{1})$. The proof is based on the results in
\cite{RZZ1}, which have been applied for various types of SPDEs with
fractional dissipation.

\begin{theorem}\label{Th3} Suppose that $\theta_1\ge 0$, $\theta_2>0$ with $\theta_1+\theta_2 > \frac{3}{4}$  and the Hypothesis (4.1) hold. Then for any
$u_0~{\in}~L^2(\Omega;\mathbb{H}^{1})$, we have local pathwise
uniqueness strong solution for \eref{spde}, and there exists a
maximal strong solution $(u,(\tau_R)_{R\in{\mathbb{N}}},\xi)$ of
\eref{spde}.
\end{theorem}

\begin{Rem}
In particular, Theorem 4.2 is applicable for the case $\theta_1=0$
and $\theta_2> \frac{3}{4}$. Hence our results generalize the
results of stochastic 3D Navier-Stokes equations ($\theta_1=0$ and
$\theta_2=1$) in \cite{GZ1}.
\end{Rem}

\begin{Rem}If $(u,(\tau_R)_{R\in{\mathbb{N}}})$ is a local strong solution for \eref{spde}, then we have
\begin{eqnarray}
{\mathbb{E}}\left(\sup_{0\leq t\leq \tau_R}\|u(t)\|
^2+\int_0^{\tau_R}\|u(t)\|_{\theta_2+1}^2dt\right)
<\infty.\label{4.10d}
\end{eqnarray}
 However, we are not able to show that \eref{4.10d} holds
when
 replacing the stopping time $\tau_R$ by any fixed (deterministic)
 $T>0$. This is the case even in the case  $\theta_1\ge 0, \theta_2 > 0$ and
$\theta_1+\theta_2\geq\frac{5}{4}$ where we can prove the existence
of global strong solution.

In the section 4.3, we will prove that \eref{4.10d} also holds for any fixed (deterministic)
 $T>0$ in the subcritical case, i.e. $\theta_1\ge 0$ and $\theta_2 > 0$ with $\theta_1+\theta_2>\frac{5}{4}$.
\end{Rem}

\begin{Rem}
We remark that the proof can be spited in two cases, namely the case
$0<\theta_2\leq1$ and the case $\theta_2>1$. For $0<\theta_2\leq1$, we can use the result in \cite{RZZ1}.
While, the framework is not adapted to $\theta_2>1$. We shall give a direct proof for this case.
\end{Rem}

Before the proof of Theorem 4.2, we introduce the following space
for later use. Let $K$ be a separable space, given
$p>1,\kappa\in(0,1)$, let $W^{\kappa,p}([0,T];K)$ be the Sobolev
space of all $u\in L^p([0,T];K)$ such that
$$\int_0^T\!\!\!\!\int_0^T\frac{|u(t)-u(s)|_K^p}{|t-s|^{1+\kappa p}}dtds<\infty,$$
endowed with the norm
$$\|u\|_{W^{\kappa,p}([0,T];K)}^p:=\int_0^T|u|_K^pdt+\int_0^T\!\!\!\!\int_0^T\frac{|u(t)-u(s)|_K^p}{|t-s|^{1+\kappa p}}dtds.$$
For the case $\kappa=1$, we take $W^{1,p}([0,T];K):=\{u\in
L^p([0,T];K);\frac{du}{dt}\in L^p([0,T];K)\}$
 with the norm
$$\|u\|_{W^{1,p}([0,T];K)}^p:=\int_0^T|u|_K^p+\left|\frac{du}{dt}\right|_K^pdt.$$
Note that for $\kappa\in(0,1)$, $W^{1,p}([0,T];K)\subset
W^{\kappa,p}([0,T];K)$ and there exists a constant $C>0$ such that
$\|u\|_{W^{\kappa,p}([0,T];K)}\leq C\|u\|_{W^{1,p}([0,T];K)}$.

 \vspace{2mm}
 {\bf Proof of Theorem 4.2:}\vspace{2mm}

\textbf{{Case 1:}} (The case $\theta_1\ge 0$, $0<\theta_2\leq1$ and $\theta_1+\theta_2>\frac{3}{4}$)

\vspace{2mm}
For this case, we shall apply Theorem 3.2 of \cite{RZZ1} to
\eref{spde}. Thanks to the assumptions for $g$, it is sufficient to check
the conditions $(b.1)$-$(b.3)$ in \cite{RZZ1} for $B(u)$.

\vspace{2mm}
(1) For $s\in[1,2]$, since $\langle{Gu}\cdot\nabla\Lambda^{s}{u},\Lambda^{s}{u}\rangle=0$, by
the H\"{o}lder's inequality we get that for $u\in{\mathbb{H}^{2s}}$,
\begin{eqnarray}
-\langle{B}(u),\Lambda^{2s}{u}\rangle\!\!\!\!\!\!\!\!&&\leq\left|\langle\Lambda^{s}\big(({Gu}\cdot\nabla){u}\big),\Lambda^{s}{u}\rangle-\langle({Gu}\cdot\nabla)\Lambda^{s}{u},\Lambda^{s}{u}\rangle\right|
\nonumber\\\!\!\!\!\!\!\!\!&&=|\langle[\Lambda^{s},Gu]\cdot\nabla
u,\Lambda^{s}{u} \rangle|
\nonumber\\\!\!\!\!\!\!\!\!&&\leq{C}\|[\Lambda^{s},Gu]\cdot\nabla
u\|_{L^p}\|\Lambda^{s}{u}\|_{L^q} .\label{4.6b}
\end{eqnarray}
Here we choose
$$\frac{1}{p}=\frac{1}{2}+\frac{\delta}{3},~~~~\frac{1}{q}=\frac{1}{2}-\frac{\delta}{3},$$
with $\delta=\delta(\theta_2)\in(0,\theta_2)$ such that
$2\theta_1+\theta_2+\delta\geq\frac{3}{2}$.

 To bound the first term of \eref{4.6b}, we make use of the
commutator estimate and get that
\begin{eqnarray}
\|[\Lambda^{s},Gu]\cdot\nabla u\|_{L^p} \leq
C\left(\|\Lambda{G}u\|_{L^{p_1}}\|\Lambda^{s}u\|_{L^{p_2}}+\|\Lambda{u}\|_{L^{p_3}}\|\Lambda^{s}{Gu}\|_{L^{p_4}}\right).\nonumber
\end{eqnarray}
Here we choose
\begin{eqnarray}\frac{1}{p_1}=\frac{\theta_2+\delta}{3},
~~\frac{1}{p_2}=\frac{1}{2}-\frac{\theta_2}{3},~~{p_3}={2},~~\frac{1}{p_4}=\frac{\delta}{3}.\label{4.7b}\end{eqnarray}
$\theta_2\in(0,1]$ implies ${p_1}\in(\frac{3}{2},\infty)$, $p_2\in(2,6]$ and $p_4\in(3,\infty)$.

 Since
$\theta_1+\theta_2>\frac{3}{4}$ and
$2\theta_1+\theta_2+\delta\geq\frac{3}{2}$, by Lemma 2.1 and Lemma
2.3, we obtain
$$\|\Lambda{Gu}\|_{L^{p_1}}\leq{C}\|Gu\|_{\frac{5}{2}-\theta_2-\delta}\leq{C}\|u\|,$$
$$\|\Lambda^s{u}\|_{L^{p_2}}\leq C\|u\|_{\theta_2+s},$$
$$\|\Lambda^s{Gu}\|_{L^{p_4}}\leq{C}\|Gu\|_{\frac{3}{2}-\delta+s}\leq
C\|u\|_{\theta_2+s},$$
$$\|\Lambda^s{u}\|_{L^{q}}\leq C\|u\|_{\delta+s}.$$
Therefore, putting the above estimates all together and by the
interpolation inequality as well as the Young's inequality, it leads
to
$$-\langle{B}(u),\Lambda^{2s}{u}\rangle\leq{C}\|u\|\|u\|_{\delta+s}\|u\|_{\theta_2+s}\leq\varepsilon\|u\|^2_{\theta_2+s}
+{C_{\varepsilon}}\|u\|^{2+2\frac{\theta_2+s-1}{\theta_2-\delta}}.$$
Thus, the coercivity condition (b.1) in \cite{RZZ1} is satisfied.

\vspace{2mm}
(2) Lemma 2.1 and Lemma 2.5 imply that, for $u\in{\mathbb{H}^{2+\theta_2}}$,
$$\|B(u)\|_{L^2}\leq{C}\|u\|\|u\|_{\theta_2+1}$$
and
\begin{eqnarray}\|\Lambda{B}(u)\|_{L^2}
\leq{C}(\|{G}u\|_{L^{p_1}}\|\Lambda^{2}u\|_{L^{p_2}}
+\|\Lambda{Gu}\|_{L^{p_3}}\|\Lambda{u}\|_{L^{p_4}}).\label{4.12}
\end{eqnarray}
Now we choose $\frac{1}{p_1}=\frac{\theta_2}{3},
~~\frac{1}{p_2}=\frac{1}{2}-\frac{\theta_2}{3}$,  so by the Sobolev
embeddings we get
$$\|{Gu}\|_{L^{p_1}}\leq{C}\|Gu\|_{\frac{3}{2}-\theta_2}\leq{C}\|u\|,$$
$$\|\Lambda^2{u}\|_{L^{p_2}}\leq C\|u\|_{\theta_2+2}.$$
For the last two terms in the right side of \eref{4.12}, if
$0<\theta_2\leq\frac{1}{2}$, we take $p_3=3,~{p_4}=6$ and the
assumption $\theta_1+\theta_2>\frac{3}{4}$ gives
$\theta_1\geq\frac{1}{4}$. So, we can get
$\|\Lambda{Gu}\|_{L^{p_3}}\leq{C}\|Gu\|_{\frac{3}{2}}\leq C\|u\|$
and $\|\Lambda{u}\|_{L^{p_4}}\leq C\|u\|_{2}\leq
C\|u\|_{\theta_2+2}.$ On the other hand, if
$\frac{1}{2}<\theta_2\leq 1$, we choose $p_3=2,~{p_4}=\infty$, Lemma
2.1 implies that $\|\Lambda{u}\|_{L^{\infty}}\leq
C\|u\|_{\theta_2+2}$.

Thus $$\|\Lambda{B}(u)\|_{L^2}\leq{C}\|u\|\|u\|_{\theta_2+2},$$
which shows that the growth condition (b.2) in \cite{RZZ1} is
proved.

 (3) For the local
monotonicity condition (b.3) in \cite{RZZ1}, since
$2\theta_1+\delta+\theta_2\geq\frac{3}{2}$, Lemma 2.2 (with $m_1=2\theta_1+\delta,m_2=0,m_3=\theta_2$) yields
\begin{eqnarray}
|\langle B(u_1)-B(u_2),u_1-u_2\rangle|\!\!\!\!\!\!\!\!
&&= |\langle B(G(u_1-u_2),u_2),u_1-u_2\rangle|\nonumber\\
&&\leq C\|G(u_1-u_2)\|_{2\theta_1+\delta}\|u_2\|\|u_1-u_2\|_{\theta_2}\nonumber\\
&&\leq C\|u_1-u_2\|_{\delta}\|u_2\|\|u_1-u_2\|_{\theta_2}\nonumber\\
&&\leq C\|u_1-u_2\|_{L^2}^{\frac{\theta_2-\delta}{\theta_2}}\|u_2\|\|u_1-u_2\|_{\theta_2}^{\frac{\theta_2+\delta}{\theta_2}}\nonumber\\
&&\leq
\varepsilon\|u_1-u_2\|^2_{\theta_2}+C_\varepsilon\|u_2\|^{\frac{2\theta_2}{\theta_2-\delta}}\|u_1-u_2\|_{L^2}^2,\label{6.2}
\end{eqnarray}
where we use Lemma 2.3 in the second inequality, the interpolation
inequality in the third inequality and the Young's inequality in the
last inequality.

 Hence all conclusions follow from Theorem 3.2 in \cite{RZZ1}. The proof is completed in this case.

\vspace{2mm} \textbf{{Case 2:} }( The case $\theta_1\ge 0$ and
$\theta_2>1$)

We first establish the existence of weak solutions to the following
equation:
\begin{eqnarray}
\left\{ \begin{aligned}
&d{u}+\nu{\Lambda}^{2\theta_2}udt+\chi_R(\|u\|)B(u)dt=\chi_R(\|u\|)g(u)dW,\\
&u(0)=u_0,
\end{aligned} \right.\label{4.41}
\end{eqnarray}
where $R>0$ is a fixed constant, and
$\chi_R:[0,\infty)\rightarrow[0,1]$ is a $C^\infty$ smooth function
such that
$$\chi_R(x)=\left\{
              \begin{array}{ll}
                1 & \hbox{for $ x \leq{R}$,} \\
                0, & \hbox{for $ x >{2R}$.}
              \end{array}
            \right.
$$

We denote by $P_n$ the projection operator onto
$\mathbb{H}_n:={\rm span}\{e^{ik\cdot{x}}:|k|\leq{n}\}$. Consider the
Galerkin approximation $u_n$ of \eref{4.41} as
\begin{eqnarray}
\left\{ \begin{aligned}
&d{u_n}+\nu{\Lambda}^{2\theta_2}u_ndt+\chi_R(\|u_n\|)P_nB(u_n)dt=\chi_R(\|u_n\|)P_ng(u_n)dW,\\
&u_n(0)=P_nu_0.
\end{aligned} \right.\label{4.42}
\end{eqnarray}
Then by the theory of SDE in finite-dimension space (see, e.g.
\cite{LR3}), there exists a unique global solution to \eref{4.42}.
According to It\^{o}'s formula, we obtain
\begin{eqnarray}
d{\|u_n\|^2}+2\nu\|u_n\|^2_{\theta_2+1}dt\!\!\!\!\!\!\!\!&&\leq-2\chi_R(\|u_n\|)\langle
B(u_n),\Lambda^2u_n\rangle dt
\nonumber\\\!\!\!\!\!\!\!\!&&~~+\chi_R(\|u_n\|)\|\Lambda
g(u_n)\|_{L_2(U,\mathbb{H}^0)}^2dt
\nonumber\\\!\!\!\!\!\!\!\!&&~~+2\chi_R(\|u_n\|)\langle \Lambda
u_n,\chi_R(\|u_n\|)\Lambda g(u_n)dW\rangle.\label{4.43}
\end{eqnarray}
By the commutator lemma, one may conclude that
\begin{eqnarray}
\langle{B}(u_n),\Lambda^{2}{u_n}\rangle\!\!\!\!\!\!\!\!&&\leq|\langle[\Lambda,Gu_n]\cdot\nabla
u^{n},\Lambda{u_n} \rangle|
\nonumber\\\!\!\!\!\!\!\!\!&&\leq{C}\|[\Lambda,Gu_n]\cdot\nabla
u_n\|_{L^{\frac{3}{2}}}\|\Lambda{u_n}\|_{L^3}
\nonumber\\\!\!\!\!\!\!\!\!&&\leq{C}\left(\|\nabla
Gu_n\|_{L^6}\|\Lambda u_n\|_{L^2}+\|\Lambda Gu_n\|_{L^6}\|\nabla
u\|_{L^2}\right)\|\Lambda{u_n}\|_{L^3}
\nonumber\\\!\!\!\!\!\!\!\!&&\leq{C}\|u_n\|_{2}\|
u_n\|\|u_n\|_{\frac{3}{2}}
\nonumber\\\!\!\!\!\!\!\!\!&&\leq\varepsilon\|u_n\|_{\theta_2+1}^2+C_\varepsilon\|u_n\|^6
,\label{4.44}
\end{eqnarray} where we use the H\"{o}lder inequality in the second inequality,  the Sobolev embedding
inequality and Lemma 2.3 in the fourth inequality, the interpolation
inequality as well as the Young's inequality in the last inequality.

 Then, $\eref{4.43}$, \eref{4.44} and the {Hypothesis} (4.1) imply
\begin{eqnarray}
d{\|u_n\|^2}+2\nu\|u_n\|^2_{\theta_2+1}dt\!\!\!\!\!\!\!\!&&\leq
[C\chi_R(\|u_n\|)\|u_n\|^2+2\varepsilon\|u_n\|^2_{\theta_2+1}+C]dt
\nonumber\\\!\!\!\!\!\!\!\!&&~~+2\chi_R(\|u_n\|)\langle \Lambda
u_n,\chi_R(\|u_n\|)\Lambda g(u_n)dW\rangle.\nonumber
\end{eqnarray}
By the BDG's inequality, it follows that
\begin{eqnarray}
&&\mathbb{E}\sup_{t\in[0,T]}{\|u_n\|^2}+\mathbb{E}\int_0^T\|u_n\|^2_{\theta_2+1}dt\nonumber\\
\leq\!\!\!\!\!\!\!\!&&
C\mathbb{E}\|u_0\|^2+CT+C\mathbb{E}\left(\int_0^T\|u_n\|^2\chi_R(\|u_n\|)\|\Lambda
g(u_n)\|^2_{L_2(U,\mathbb{H}^0)}dt\right)^{1/2}
\nonumber\\\leq\!\!\!\!\!\!\!\!&& C_T,\label{4.45}
\end{eqnarray}
where $C_T$ is a constant independent of $n$.

Now, we prove that the family $\mathcal{L}(u_n)_{n\in\mathbb{N}}$ is
tight in $C([0,T];\mathbb{H}^{1-\theta_2})$. Here,
$\mathcal{L}(u_n)$ means the law of $u_n$. By \eref{4.45}, for each
$t\in[0,T]$,  $\mathcal{L}(u_n(t))$ is tight on
$\mathbb{H}^{1-\theta_2}$. Then according to Aldous's criterion in
\cite{Al}, it suffices to check that for all stopping times
$\tau_n\leq{T}$ and $\eta_n\rightarrow0$,
\begin{eqnarray}
\lim_n\mathbb{E}\|u_n(\tau_n+\eta_n)-u_n(\tau_n)\|_{1-\theta_2}=0.\label{4.46}
\end{eqnarray}
Note that
\begin{eqnarray}
u_n(\tau_n+\eta_n)-u_n(\tau_n)=\!\!\!\!\!\!\!\!\!\!&&~-\int_{\tau_n}^{\tau_n+\eta_n}(\nu{\Lambda}^{2\theta_2}u_n+\chi_R(\|u_n\|)P_nB(u_n)\dt
\nonumber\\\!\!\!\!\!\!\!\!\!\!&&~{+}\int_{\tau_n}^{\tau_n+\eta_n}\chi_R(\|u_n\|)P_ng(u_n)dW.\label{4.47}
\end{eqnarray}
By \eref{4.45}, we get that for large $n$
$$\mathbb{E}\left\|\int_{\tau_n}^{\tau_n+\eta_n}\nu{\Lambda}^{2\theta_2}u_ndt\right\|_{1-\theta_2}
\leq{C}{\eta_n}^{1/2}\left(\mathbb{E}\int_0^{T+1}\|u_n\|^2_{\theta_2+1}dt\right)^{1/2}\rightarrow0,~~\text{as}~\eta_n\rightarrow0.$$
Thanks to Lemmas 2.2 and 2.3, it is easy to get
\begin{eqnarray}
\|{B}(u_n)\|_{1-\theta_2}\leq
C\|{B}(u_n)\|_{L^2}\leq{C}\|u_n\|\|u_n\|_{\theta_2+1}.\label{4.52}
\end{eqnarray}
From \eref{4.52} we infer that
\begin{eqnarray}\mathbb{E}\left\|\int_{\tau_n}^{\tau_n+\eta_n}\chi_R(\|u_n\|)P_nB(u_n)dt\right\|_{1-\theta_2}
\!\!\!\!\!\!\!\!\!\!&&\leq{C}\mathbb{E}\int_{\tau_n}^{\tau_n+\eta_n}\|u_n\|_{\theta_2+1}dt
\nonumber\\\!\!\!\!\!\!\!\!\!\!&&\leq{C}{\eta_n}^{1/2}\left(\mathbb{E}\int_0^{T+1}\|u_n\|^2_{\theta_2+1}dt\right)^{1/2}\rightarrow0,
\nonumber\\\!\!\!\!\!\!\!\!\!\!&&\text{as}~\eta_n\rightarrow0.\nonumber
\end{eqnarray}
Similarly, we obtain by {Hypothesis} (4.1)
\begin{eqnarray}\mathbb{E}\left\|\int_{\tau_n}^{\tau_n+\eta_n}\chi_R(\|u_n\|)P_ng(u_n)dW\right\|_{1-\theta_2}^2
\!\!\!\!\!\!\!\!\!\!&&\leq{C}\mathbb{E}\int_{\tau_n}^{\tau_n+\eta_n}\chi_R(\|u_n\|)\|\Lambda
g(u_n)\|^2_{L_2(U,\mathbb{H}^0)}dt
\nonumber\\\!\!\!\!\!\!\!\!\!\!&&\leq{C}{\eta_n}\rightarrow0,~~\text{as}~\eta_n\rightarrow0.\nonumber
\end{eqnarray}
Thus, \eref{4.46} follows, which implies the tightness of
$\mathcal{L}(u_n)_{n\in\mathbb{N}}$ in
$C([0,T];\mathbb{H}^{1-\theta_2})$.

We also make use of a variation of the BDG's inequality (see, Lemma
2.1 in \cite{FG}) and get that for $\kappa\in[0,1/2)$,
\begin{eqnarray}
\mathbb{E}\left\|\int_0^t\chi_R(\|u_n\|)
g(u_n)dW\right\|_{W^{\kappa,2}([0,T];\mathbb{H}^0)}^2dt\leq\!\!\!\!\!\!\!\!\!\!&&~
C\mathbb{E}\int_0^T\chi_R(\|u_n\|)\|
g(u_n)\|^2_{L_2(U,\mathbb{H}^0)}dt
\nonumber\\\leq\!\!\!\!\!\!\!\!\!\!&&
~C\mathbb{E}\int_0^T\chi_R(\|u_n\|)\| \Lambda
g(u_n)\|^2_{L_2(U,\mathbb{H}^0)}dt
\nonumber\\\leq\!\!\!\!\!\!\!\!\!\!&&~C_{1,T}.\label{4.51}
\end{eqnarray} By \eref{4.52}, we conclude
\begin{eqnarray}
&&\mathbb{E}\left\|u_n(t)-\int_0^t\chi_R(\|u_n\|)
g(u_n)dW\right\|_{W^{1,2}([0,T];\mathbb{H}^{1-\theta_2})}^2
\nonumber\\\leq\!\!\!\!\!\!\!\!&&
C\mathbb{E}\|u_0\|_{L^2}^2+C\mathbb{E}\int_0^T(\|{\Lambda}^{2\theta_2}u_n\|_{1-\theta_2}^2+\chi_R(\|u_n\|)\|B(u_n)\|_{L^2}^2dt
\nonumber\\\leq\!\!\!\!\!\!\!\!&&
C\mathbb{E}\|u_0\|_{L^2}^2+C\mathbb{E}\int_0^T\|u_n\|_{\theta_2+1}^2dt
\nonumber\\\leq\!\!\!\!\!\!\!\!&& C_{2,T}.\label{4.52b}
\end{eqnarray}
$\eref{4.45}$, \eref{4.51} and \eref{4.52b} imply that the laws
$\mathcal{L}(u_n)_{n\in\mathbb{N}}$ are bounded in probability in
$$L^2([0,T];\mathbb{H}^{\theta_2+1})\cap{W^{\kappa,2}([0,T];\mathbb{H}^{1-\theta_2})}.$$
Thus,  by \cite[Theorem 2.1]{FG} we get that
$\mathcal{L}(u_n)_{n\in\mathbb{N}}$ is tight in
$L^2([0,T];\mathbb{H}^1)$. Therefore, there exists a subsequence,
still denoted by $u_n$, such that $\mathcal{L}(u_n)$ converges
weakly in $L^2([0,T];\mathbb{H}^1)\cap
C([0,T];\mathbb{H}^{1-\theta_2})$. The Skorohod's embedding theorem
yields that there exists a stochastic basis $(\Omega^1,
\mathscr{F}^1,\{\mathscr{F}^1_t\}_{t\geq0},\mathbb{P}^1)$ and
$L^2(0,T;\mathbb{H}^1)\cap C([0,T];\mathbb{H}^{1-\theta_2})$-valued
random variables $u^1$ and $u_n^1,n\geq1$ on it, such that $u^1_n$
and $u_n$ have the same law and $u^1_n\rightarrow{u^1}$ in
$L^2([0,T];\mathbb{H}^1)\cap C([0,T];\mathbb{H}^{1-\theta_2})$,
$\mathbb{P}^1$-a.s. For $u_n^1$ we also have \eref{4.45}. This and
the Fatou's lemma imply
\begin{eqnarray*}u^1\in
L^2\big(\Omega^1,L^2([0,T];\mathbb{H}^{\theta_2+1})\cap
L^\infty([0,T];\mathbb{H}^{1})\big).
\end{eqnarray*}

For $n\geq1$, define the $\mathbb{H}_n$-valued process
$$M_n^1(t):=u_n^1(t)-P_nu_0+\int_0^t\left(\nu{\Lambda}^{2\theta_2}u_n^1+\chi_R(\|u_n^1\|)P_nB(u_n^1)\right)ds.$$
Then $M_n^1(t)$ is  a square integrable martingale with respect to
the filtration $\{\mathscr{G}_n^1\}_t=\sigma\{u_n^1(s),s\leq t\}$
with quadratic variation process
$$\langle M_n^1\rangle_t=\int_0^t\chi_R(\|u_n^1\|)^2P_ng(u_n^1(s))g(u_n^1(s))^*P_nds.$$
For $s\leq t\in[0,T]$, bounded continuous functions $\phi$ on $L^2([0,s];\mathbb{H}^1)\cap C([0,s];\mathbb{H}^{1-\theta_2})$
and $v\in\cap_{l=1}^\infty\mathbb{H}^{l}$, we have
\begin{eqnarray}{\mathbb{E}^1}(\langle
M_n^1(t)-M_n^1(s),v\rangle)\phi(u_n^1|_{[0,s]}))=0\label{4.53}
\end{eqnarray}
and
\begin{eqnarray}{\mathbb{E}^1}\left[\left(\langle M_n^1(t),v\rangle^2-\langle M_n^1(s),v\rangle^2-
\int_s^t\chi_R(\|u_n^1\|)^2)\|g(u_n^1)^*P_nv\|_U^2dr\right)\phi(u_n^1|_{[0,s]})\right]=0.\label{4.54}
\end{eqnarray}
In order to take the limit in \eref{4.53} as $n\rightarrow\infty$,
we estimate
\begin{eqnarray}&&\int_0^t|\langle\chi_R(\|u_n^1\|)P_nB(u_n^1),v\rangle-\langle\chi_R(\|u^1\|)B(u^1),v\rangle|ds\nonumber\\
\leq\!\!\!\!\!\!\!\!&& \int_0^t|\langle\chi_R(\|u_n^1\|)(B(u_n^1)-B(u^1)),P_nv\rangle|ds+\int_0^t|\langle\chi_R(\|u_n^1\|)B(u^1),(I-P_n)v\rangle|ds\nonumber\\
\!\!\!\!\!\!\!\!\!&&+\int_0^t|\langle(\chi_R(\|u_n^1\|)-\chi_R(\|u^1\|))B(u^1),v\rangle|ds\nonumber\\
=:\!\!\!\!\!\!\!\!&& I_1+I_2+I_3.\label{4.54a}
\end{eqnarray}
For $I_1$, by the bilinearity of $B$, we have
$$B(u_n^1)-B(u^1)=B(Gu_n^1-Gu^1,u_n^1)+B(Gu^1,u_n^1-u^1),$$
which together with Lemma 2.2 implies that
\begin{eqnarray}I_1\!\!\!\!\!\!\!\!&&\leq C\|v\|
\int_0^t(\|u_n^1\|+\|u^1\|)\|u_n^1-u^1\|ds
\nonumber\\\!\!\!\!\!\!\!\!&&\leq{C}\|v\|
\left(\int_0^t\|u_n^1\|^2+\|u^1\|^2ds\right)^{1/2}\left(\int_0^t\|u_n^1-u^1\|^2ds\right)^{1/2}
\nonumber\\\!\!\!\!\!\!\!\!&&\rightarrow0,~\text{for}~\mathbb{P}^1{\text -a.e.}~\omega\in\Omega^1.\label{4.54b}
\end{eqnarray}
For $I_2$, we have
\begin{eqnarray}I_2\leq{C}\|(I-P_n)v\|\int_0^t\|u^1\|^2ds\rightarrow0,~\text{for}~\mathbb{P}^1{\text -a.e.}~\omega\in\Omega^1.\label{4.54c}
\end{eqnarray}
By the dominated convergence theorem, we can get the last term
$I_3\rightarrow0,$ for almost every $\omega\in\Omega^1$. Thus,
combining \eref{4.54a}-\eref{4.54c}, we infer that for almost every
$\omega\in\Omega^1$
\begin{eqnarray}
\int_0^t\langle\chi_R(\|u_n^1\|)P_nB(u_n^1),v\rangle\rightarrow
\int_0^t\langle\chi_R(\|u^1\|)B(u^1),v\rangle ds.\label{4.54d}
\end{eqnarray}
Applying the BDG's inequality, we get that for any $p\geq2$
\begin{eqnarray}\sup_n\mathbb{E}^1|\langle M_n^1(t),v\rangle|^{2p}\leq C\sup_n\mathbb{E}^1\left(
\int_0^t\chi_R(\|u_n^1\|)^2\|g(u_n^1)^*P_nv\|_U^2dr\right)^p<\infty.\label{4.53a}
\end{eqnarray}
Note that the convergence for the linear term is direct, by
\eref{4.54d} and \eref{4.53a} we can get that for
$v\in\cap_{l=1}^\infty\mathbb{H}^{l}$,
$$\lim_{n\rightarrow\infty}\mathbb{E}^1|\langle M_n^1(t)-M^1(s),v\rangle|=0$$
and
$$\lim_{n\rightarrow\infty}\mathbb{E}^1|\langle M_n^1(t)-M^1(s),v\rangle|^2=0,$$
where
$$M^1(t):=u^1(t)-u^1(0)+\int_0^t\left(\nu{\Lambda}^{2\theta_2}u^1+\chi_R(\|u^1\|)B(u^1)\right)ds.$$

Taking the limit in \eref{4.53} and \eref{4.54}, we derive that for
 $s\leq t\in[0,T]$, bounded continuous functions $\phi$ on
$L^2([0,s];\mathbb{H}^1)\cap C([0,s];\mathbb{H}^{1-\theta_2})$ and
$v\in\cap_{l=1}^\infty\mathbb{H}^{l}$, we have
\begin{eqnarray}{\mathbb{E}^1}(\langle
M^1(t)-M^1(s),v\rangle)\phi(u^1|_{[0,s]}))=0\nonumber
\end{eqnarray}
and
\begin{eqnarray}{\mathbb{E}^1}\left(\left(\langle M^1(t),v\rangle^2-\langle M^1(s),v\rangle^2-
\int_s^t\chi_R(\|u^1\|)^2\|g(u^1)^*v\|_U^2dr\right)\phi(u^1|_{[0,s]})\right)=0.\nonumber
\end{eqnarray}
Thus, according to the martingale representation theorem (cf.
\cite[Theorem 8.2]{DZ}), there exists a stochastic basis
$(\tilde{\Omega},
\tilde{\mathscr{F}},\{\tilde{\mathscr{F}}_t\}_{t\geq0},
\tilde{\mathbb{P}})$, a cylindrical Wiener process $\tilde{W}$ and
an $\tilde{\mathscr{F}}_t$-adapted process $\tilde{u}$ with path in
$L^2([0,T];\mathbb{H}^1)\cap C([0,T];\mathbb{H}^{1-\theta_2})$ such
that $\tilde{u}$ satisfies \eref{4.41} with $W$ replaced by
$\tilde{W}$ and $\tilde{u}_0$ has the same distribution as $u_0$.
For $\tilde{u}$ we also have \eref{4.45}, hence, it follows that
\begin{eqnarray}\tilde{u}\in
L^2\big(\tilde{\Omega},L^2([0,T];\mathbb{H}^{\theta_2+1})\cap
L^\infty([0,T];\mathbb{H}^{1})\big).\label{4.58}
\end{eqnarray}

Now we want to show that $\tilde{u}\in C([0,T];\mathbb{H}^1)$ a.s.,
which is needed in order to justify the following stopping time
\eref{4.57b} is well defined. To this end we define
\begin{eqnarray}
dz+\nu\Lambda^{2\theta_2}z=\chi_R(\|\tilde{u}\|)g(\tilde{u})d\tilde{W},~~~z(0)=0.\label{4.55}
\end{eqnarray}
Since $\chi_R(\|\tilde{u}\|)g(\tilde{u})\in
L^2(\tilde{\Omega},L^2([0,T],L_2(U,\mathbb{H}^1)))$, we have
\begin{eqnarray}z\in L^2(\tilde{\Omega},C([0,T];\mathbb{H}^1))\cap L^2(\tilde{\Omega},L^2([0,T];\mathbb{H}^{1+\theta_2})).
\label{4.56}\end{eqnarray} Take $\bar{u}=\tilde{u}-z$. Subtracting
\eref{4.55} from \eref{4.41}, then we see that $\bar{u}$ solves
\begin{eqnarray}
\left\{ \begin{aligned}
&\frac{d\bar{u}}{dt}+\nu\Lambda^{2\theta_2}\bar{u}+\chi_R(\|\bar{u}+z\|)B(\bar{u}+z)=0,\\
&\bar{u}(0)=\tilde{u}_0,
\end{aligned} \right.\label{4.57}
\end{eqnarray}
which is a (pathwise) deterministic PDE.

Due to \eref{4.58} and \eref{4.56}, we infer that $\bar{u}\in
L^2\big(\tilde{\Omega},L^2([0,T];\mathbb{H}^{\theta_2+1})\cap
L^\infty([0,T];\mathbb{H}^{1})\big)$. Hence, we have
$$\Lambda^{2\theta_2}\bar{u},\chi_R(\|\bar{u}+z\|)B(\bar{u}+z)\in L^2(\tilde{\Omega},L^2([0,T];\mathbb{H}^{1-\theta_2})).$$
We conclude with \eref{4.57} that
$$\frac{d\bar{u}}{dt}\in L^2(\tilde{\Omega},L^2([0,T];\mathbb{H}^{1-\theta_2})),~~~\bar{u}\in L^2(\tilde{\Omega},L^2([0,T];\mathbb{H}^{\theta_2+1})).$$
Applying the strong continuity result (see \cite[Chaper 3, Lemma
1.2]{RT} or \cite[Lemma 6]{BF}), we infer that $\bar{u}\in
C([0,T;\mathbb{H}^1)$, a.s. This and \eref{4.56} imply that
$\tilde{u}\in C([0,T];\mathbb{H}^1)$ a.s.

Define the stopping time
\begin{eqnarray}\tau_R:=\inf\{t\geq0;\|\tilde{u}(t)\|\geq
R\}.\label{4.57b}
\end{eqnarray}
Then $(\tilde{u},\tau_R)$ is a local weak solution of \eref{spde}
such that $\tilde{u}(\cdot\wedge\tau_R)\in
C([0,\infty);\mathbb{H}^1)$ a.s., and
$\tilde{u}(\cdot\wedge\tau_R)\in
L^2(\tilde{\Omega},L^2_{loc}([0,\infty),\mathbb{H}^{\theta_2+1}))$.

To complete the proof of the case, it remains to prove the pathwise
uniqueness and apply the Yamada-Watanable theorem (cf \cite[Threorem
3.14]{K}). These technical details are similar to the arguments in
\cite{RZZ1}, so we omit further details.

Now the proof of Theorem 4.2 is completed.\hspace{\fill}$\Box$

\subsection{Global existence}
In the deterministic case ($g\equiv0$), the global solutions of
Leray-regularized equations with fractional dissipation have been
intensively investigated (cf. \cite{A,BMR,BF,OT,P,Y}). In
particular, Barbato, Morandin and Romito in \cite{BMR} showed the
existence of a smooth global solution to the 3D Leray-$\alpha$ model
with the logarithmical dissipation when $\theta_1+\theta_2 \geq
\frac{5}{4}$ (with $\theta_1,\theta_2\geq0$). The result in
\cite{BMR} was obtained by analyzing energy dispersion over dyadic
shells, which is a completely different approach with the one
employed in this work.


The existence of the noise perturbation makes the problem more
interesting and challenging. In this section, inspired by
\cite{DGTZ,GZ1}, we show that the strong solution of \eref{spde} is
global when $\theta_1\ge 0$, $\theta_2 > 0$ and $\theta_1+\theta_2
\geq \frac{5}{4}$. By Theorem 4.2, let
$(u,(\tau_R)_{R\in{\mathbb{N}}},\xi)$ be a maximal strong solution,
if $\xi(\omega)<\infty$ for $\omega\in\Omega$, then the
$\mathbb{H}^{1}$ norm of the solution must blow up at this maximal
time as expressed by \eref{4.3c}. Next lemma establishes some
estimates for the solution of \eref{spde}, which do not depend on
$\tau_R$ and are useful for the proof of global existence.

\begin{lemma}
Under the assumptions of Theorem 4.1, if
$(u,(\tau_R)_{R\in{\mathbb{N}}},\xi)$ is a maximal strong solution
of \eref{spde}, then for any $t>0$
\begin{eqnarray}{\mathbb{E}}\left(\sup_{0\leq s\leq\xi\wedge{t}}\|u(s)\|_{L^2}^2+\int_0^{\xi\wedge{t}}\|u(s)\|_{\theta_2}^2ds\right)<\infty.\label{4.6a}
\end{eqnarray}
\end{lemma}

\begin{proof}
Taking It\^{o}'s formula, we obtain that, for any $t>0$,
$s\in[0,\tau_R\wedge{t}]$
\begin{eqnarray}
\|u(s)\|_{L^2}^2=\!\!\!\!\!\!\!\!&&\|u_0\|_{L^2}^2-2\int_0^s\left(\langle\nu{\Lambda}^{2\theta_2}u(r)+B(u(r)),u(r)\rangle
\right)dr\nonumber\\&&+\int_0^s
\|{g}(u(r))\|_{L_2(U,\mathbb{H}^0)}^2dr+2\int_0^s\langle
{g}(u(r)d{W}(r),u(r)\rangle.\label{4.7}
\end{eqnarray}

By the Hypothesis (4.2) and the fact that $\langle B(u),u\rangle=0
$, we have
\begin{eqnarray}
&&\|u(s)\|_{L^2}^2+\int_0^s\|u(r)\|_{\theta_2}^2dr\nonumber\\\leq
\!\!\!\!\!\!\!\!&&{C}+\|u_0\|_{L^2}^2+C\int_0^s\|u(r)\|_{L^2}^2dr
+2\int_0^s\langle {g}(u(r))d{W}(r),u(r)\rangle.\label{4.8a}
\end{eqnarray}
The BDG's inequality and the Young's inequality yield
\begin{eqnarray}
&&{\mathbb{E}}\sup_{s\in[0,\tau_R\wedge{t}]}\left|\int_0^s\langle
{g}(u(r))d{W}(r),u(r)\rangle\right|\nonumber\\
\leq\!\!\!\!\!\!\!\!&& 3{\mathbb{E}}\left(\int_0^{\tau_R\wedge{t}}\|u(r)\|_{L^2}^2\|{g}(u(r))\|_{L_2(U,\mathbb{H}^0)}^2dr\right)^{1/2}\nonumber\\
\leq\!\!\!\!\!\!\!\!&&
3{\mathbb{E}}\left(\sup_{s\in[0,\tau_R\wedge{t}]}\|u(s)\|_{L^2}^2\cdot{C}\int_0^{\tau_R\wedge{t}}(1+\|u(s)\|_{L^2}^2)ds
\right)^{1/2}\nonumber\\
\leq\!\!\!\!\!\!\!\!&&
\varepsilon\mathbb{E}\sup_{s\in[0,\tau_R\wedge{t}]}\|u(s)\|_{L^2}^2+C_{\varepsilon}
{\mathbb{E}}\int_0^{\tau_R\wedge{t}}(1+\|u(s)\|_{L^2}^2)ds,\label{4.9a}
\end{eqnarray}
where $\varepsilon>0$ is a small constant.

By \eref{4.8a} and \eref{4.9a}, one deduces
\begin{eqnarray}
&&{\mathbb{E}}\sup_{s\in[0,\tau_R\wedge{t}]}\|u(s)\|_{L^2}^2+{\mathbb{E}}\int_0^{\tau_R\wedge{t}}\|u(s)\|_{\theta_2}^2ds\nonumber\\
\leq\!\!\!\!\!\!\!\!&&
C+C{\mathbb{E}}\|u_0\|_{L^2}^2+{\mathbb{E}}\int_0^{\tau_R\wedge{t}}\|u(s)\|_{L^2}^2ds.\label{4.7a}
\end{eqnarray}
 The Gronwall's lemma implies that
\begin{eqnarray}{\mathbb{E}}\sup_{0\leq s\leq \tau_R\wedge{t}}\|u(s)\|_{L^2}^2+{\mathbb{E}}\int_0^{\tau_R\wedge{t}}\|u(s)\|_{\theta_2}^2ds\leq C(1+{\mathbb{E}}\|u_0\|_{L^2}^2),\label{4.6c}
\end{eqnarray}
where $C=C(u_0,T)$ is a constant independent of $R$.

Thus \eref{4.6a} can be proved by the monotone convergence theorem
as
 $R\rightarrow\infty$. The proof is completed. \hspace{\fill}$\Box$
\end{proof}

\subsection*{ Proof of Theorem 4.1:}

For any $K,M>0$ we define the following stopping times

$$\rho_{M}:=\inf_{t\geq0}\left\{\sup_{s\in[0,t]}\|u(s)\|^2+\int_0^t\|u(s)\|_{\theta_2+1}^{2}ds\geq{M}\right\}\wedge\xi$$
and
$$\gamma_K:=\inf_{t\geq0}\left\{\int_0^{t\wedge\xi}\|u(s)\|^2_{\theta_2}ds\geq K\right\}.$$
 Here we take $\inf\emptyset=\infty$. It is clear that $(u,\rho_{M})$ is a local strong solution.

Applying $\Lambda$ to equation \eref{spde} and taking It\^{o}'s
formula, for any $t>0$, $s\in[0,\gamma_K\wedge\rho_M\wedge{t}]$, we
have
\begin{eqnarray}
&&\|u(s)\|^2+2\nu\int_0^s\|u(r)\|_{\theta_2+1}^{2}dr\nonumber\\
=\!\!\!\!\!\!\!\!&&\|u_0\|^2 -2\int_0^s\langle
\Lambda{B(u(r))},\Lambda{u}(r)\rangle dr
+\int_0^s\|\Lambda{g}(u(r))\|_{L_2(U,\mathbb{H}^0)}^2dr
\nonumber\\&&+2\int_0^s\langle\Lambda{{g}(u(r))}d{W}(r),\Lambda{u}(r)\rangle.\label{4.11}
\end{eqnarray}
For any stopping times
$0\leq\tau_a\leq\tau_b\leq\gamma_K\wedge\rho_M\wedge{t}$, taking a
supremum over the interval $[{\tau_a},{\tau_b}]$ and then taking
expectation with respect to the resulting expression, we deduce that
\begin{eqnarray}
&&{\mathbb{E}}\left(\sup_{s\in[{\tau_a},{\tau_b}]}\|u(s)\|^2+\nu\int_{\tau_a}^{\tau_{b}}\|u(r)\|_{\theta_2+1}^{2}dr\right)\nonumber\\
\leq\!\!\!\!\!\!\!\!\!&&
C_0\mathbb{E}\left(\|u(\tau_a)\|^2+\int_{\tau_a}^{\tau_{b}}\left(|\langle
\Lambda{B(u(r))},\Lambda{u}(r)\rangle|+(1+\|{u}(r)\|^2)\right)dr\right)\nonumber\\
&&+C_0
{\mathbb{E}}\sup_{s\in[{\tau_a},{\tau_b}]}\left|\int_{\tau_a}^s\langle\Lambda{g}(u(r))d{W}(r),{\Lambda}u(r)\rangle\right|,
\label{4.8f}
\end{eqnarray}
where we use the Hypothesis (4.2), and $C_0$ is a constant
independent of $\tau_a$ and $\tau_b$.

Similar as the proof of Theorem 4.2, by the commutator lemma, we
have
\begin{eqnarray}
|\langle\Lambda{B(u)},\Lambda{u}\rangle|\leq
C(\|\Lambda{Gu}\|_{L^{p_1}}\|\Lambda{u}\|_{L^{p_2}})\|\Lambda{u}\|_{L^2},\nonumber
\end{eqnarray}
where $p_1$ and $p_2$ are given by

$$\left\{
                  \begin{array}{ll}
              \frac{1}{p_1}=\frac{\theta_2}{3},\frac{1}{p_2}=\frac{1}{2}-\frac{\theta_2}{3} & \text{if}~~0<\theta_2<\frac{3}{2}; \\
                  p_1=6,~~ p_2=3 & \text{if}~~\theta_2\geq\frac{3}{2}.
                  \end{array}
                \right.
$$

 Thanks to the
fact $\theta_1+\theta_2\geq\frac{5}{4}$, by Lemmas 2.1 and 2.3, it
is not difficult to get that for $0<\theta_2<\frac{3}{2}$,
$$\|\Lambda{u}\|_{L^{p_2}}\leq C\|u\|_{\theta_2+1},~~~~\|\Lambda{Gu}\|_{L^{p_1}}\leq{C}\|Gu\|_{\frac{5}{2}-\theta_2}\leq
C\|u\|_{\theta_2}.$$ Similarly, for $\theta_2\geq\frac{3}{2}$ we
also have
$$\|\Lambda{u}\|_{L^{p_2}}\leq C\|u\|_{\frac{3}{2}}\leq
C\|u\|_{\theta_2},~~~~\|\Lambda{Gu}\|_{L^{p_1}}\leq{C}\|Gu\|_{2}\leq
C\|u\|_{\theta_2+1}.$$

Therefore, putting the above estimates all together and applying the
Young's inequality, it leads to
\begin{eqnarray}
|\langle\Lambda{B(u)},\Lambda{u}\rangle|\leq{C}\|u\|_{\theta_2}\|u\|\|u\|_{\theta_2+1}
\leq\varepsilon\|u\|_{\theta_2+1}^{2}+C_{\varepsilon}\|u\|_{\theta_2}^{2}\|u\|^2,\label{4.11a}
\end{eqnarray}
where $\varepsilon>0$ is a small constant.

 By the BDG's inequality and
Young's inequality, we have
\begin{eqnarray}
&&{\mathbb{E}}\sup_{s\in[\tau_a,\tau_{b}]}\left|\int_{\tau_a}^s\langle\Lambda{g}(u(r))d{W}(r),{\Lambda}u(r)\rangle\right|\nonumber\\
\leq\!\!\!\!\!\!\!\!&&3{\mathbb{E}}\left(\int_{\tau_a}^{\tau_{b}}\|u(s)\|^{2}\|\Lambda{g}(u(s))\|_{L_2(U,\mathbb{H}^0)}^2ds
\right)^{1/2}\nonumber\\
\leq\!\!\!\!\!\!\!\!&&3{\mathbb{E}}\left(\sup_{s\in[{\tau_a},{\tau_b}]}\|u(s)\|^{2}\cdot{C}\int_{\tau_a}^{\tau_{b}}(1+\|u(s)\|^2)ds
\right)^{1/2}\nonumber\\
\leq\!\!\!\!\!\!\!\!&&\varepsilon{\mathbb{E}}\sup_{s\in[{\tau_a},{\tau_b}]}\|u(s)\|^2+C_{\varepsilon}
{\mathbb{E}}\int_{\tau_a}^{\tau_b}(1+\|u(s)\|^2)ds.\label{4.9c}
\end{eqnarray}
 Combining the estimates
\eref{4.8f}-\eref{4.9c}, we conclude that
\begin{eqnarray}
&&{\mathbb{E}}\left(\sup_{s\in[{\tau_a},{\tau_b}]}\|u(s)\|^2+\int_{\tau_a}^{\tau_{b}}\|u(r)\|_{\theta_2+1}^{2}dr\right)\nonumber\\
\leq\!\!\!\!\!\!\!\!&&
C\mathbb{E}\left(\|u(\tau_a)\|^2+\int_{\tau_a}^{\tau_{b}}(1+\|u(r)\|_{\theta_2}^{2})\|u(r)\|^2dr\right),
\label{4.8c}
\end{eqnarray}
where $C$ is independent of $\tau_a$ and $\tau_b$. Now, we can apply
the stochastic version Gronwall's lemma from \cite[Lemma 5.3]{GZ1},
which is recalled in the appendix as Lemma 5.1. Note that by
definition of $\gamma_k$,
$$\int_0^{\gamma_k}\|u\|_{\theta_2}^{2}ds\leq K,\ \ \text{a.s.}$$
Taking $X:=\|u\|^2$, $Y:=\|u\|_{\theta_2+1}^2$,
$R:=1+\|u\|_{\theta_2}^{2}$, $Z:=0$ and
$\tau:=\gamma_K\wedge\rho_M\wedge{t}$ in Lemma 5.1, we therefore get
that
\begin{eqnarray}
{\mathbb{E}}\left(\sup_{0\leq{s}\leq\gamma_K\wedge\rho_M\wedge{t}}\|u(s)\|^2+\int_0^{\gamma_K\wedge\rho_M\wedge{t}}\|u(s)\|_{\theta_2+1}^{2}ds\right)
\leq C_{c_0,t,K}(1+\mathbb{E}\|u_0\|^2),\label{4.10}
\end{eqnarray}
where  $C_{c_0,t,K}$ is a constant independent of $M$. Then we have
that, for any $t>0$
\begin{eqnarray}
&&\mathbb{P}(\rho_M\leq t)\nonumber\\
\leq\!\!\!\!\!\!\!\!&&\mathbb{P}\left(\{\rho_M\leq t\}\cap\{\gamma_K>t\}\right)+\mathbb{P}(\gamma_K\leq t)\nonumber\\
\leq\!\!\!\!\!\!\!\!&&\mathbb{P}\left(\left\{\sup_{0\leq{s}\leq\rho_M\wedge{t}}\|u(s)\|^2
+\int_0^{\rho_M\wedge{t}}\|u(s)\|_{\theta_2+1}^{2}ds\geq M\right\}\cap\{\gamma_K>t\}\right)+\mathbb{P}(\gamma_K\leq t)\nonumber\\
\leq\!\!\!\!\!\!\!\!&&\mathbb{P}\left(\sup_{0\leq{s}\leq\gamma_K\wedge\rho_M\wedge{t}}\|u(s)\|^2
+\int_0^{\gamma_K\wedge\rho_M\wedge{t}}\|u(s)\|_{\theta_2+1}^{2}ds\geq M\right )+\mathbb{P}(\gamma_K\leq t)\nonumber\\
\leq\!\!\!\!\!\!\!\!&&\frac{1}{M}\mathbb{E}\left(\sup_{0\leq{s}\leq\gamma_K\wedge\rho_M\wedge{t}}\|u(s)\|^2
+\int_0^{\gamma_K\wedge\rho_M\wedge{t}}\|u(s)\|_{\theta_2+1}^{2}ds\right
)+\mathbb{P}(\gamma_K\leq
t)\nonumber\\
\leq\!\!\!\!\!\!\!\!&&\frac{C_{c_0,t,K}(1+\mathbb{E}\|u_0\|^2)}{M}+\mathbb{P}(\gamma_K\leq
t).\nonumber
\end{eqnarray}
Thus, for any fixed $K>0$,
$$\lim_{M\rightarrow\infty}\mathbb{P}(\rho_M\leq
t)\leq\mathbb{P}(\gamma_K\leq t).$$
By Lemma 4.1 we obtain that
$$\mathbb{P}(\gamma_K\leq t)\leq\mathbb{P}\left(\int_0^{t\wedge\xi}\|u(s)\|^2_{\theta_2}ds\geq K\right)\leq\frac{1}{K}\mathbb{E}\left(
\int_0^{t\wedge\xi}\|u(s)\|_{\theta_2}^{2}ds\right),$$ which goes to
zero as $K\rightarrow\infty$. Hence, we get that for any $t>0$,
$$\lim_{M\rightarrow\infty}\mathbb{P}(\rho_M\leq t)=0.$$

Let $M\rightarrow\infty$ in \eref{4.10}, by the monotone convergence
theorem, we can get that for any $t>0$,
\begin{eqnarray}
{\mathbb{E}}\left(\sup_{0\leq{s}\leq\gamma_K\wedge{t}}\|u(s)\|^2+\int_0^{\gamma_K\wedge{t}}\|u(s)\|_{\theta_2+1}^{2}ds\right)
\leq C_{c_0,t,K}(1+\mathbb{E}\|u_0\|^2).\label{4.10a}
\end{eqnarray}

We now want to show that for any $K>0$, $\gamma_K\leq\xi$ a.s.
Suppose $\mathbb{P}(\gamma_K>\xi)>0$. Denoted by $\mathbb{Q}^+$ the
set of negative rational numbers, we have
$\{\gamma_K>\xi\}=\bigcup_{t\in\mathbb
Q^+}\{\gamma_K\wedge{t}>\xi\}$. Hence there is a
$t_0\in\mathbb{Q}^+$ such that
$\mathbb{P}(\gamma_K\wedge{t_0}>\xi)>0$. By the definition of $\xi$
(see \eref{4.3c}), we would infer that
$$\sup_{0\leq{s}\leq\gamma_K\wedge{t_0}}\|u(s)\|^2+\int_0^{\gamma_K\wedge{t_0}}\|u(s)\|_{\theta_2+1}^{2}ds\geq\sup_{0\leq{s}
\leq\xi}\|u(s)\|^2=\infty$$ on $\{\gamma_K\wedge{t_0}>\xi\}$. Hence,
we have a contradiction with \eref{4.10a}. This  means that
$\gamma_K\leq\xi$ a.s. for any $K>0$.

According to Lemma 4.1, it is obvious that
$$\lim_{K\rightarrow\infty}\gamma_K=\infty~~~\text{a.s.}$$
Consequently, $\xi=\infty$ a.s., and the solution $u$ is global in
the sense of definition 4.1. We complete the proof of Theorem 4.1.
\hspace{\fill}$\Box$

\subsection{Moments of the solution}
As a continuation of Remark 4.3, we establish some
moment estimates of the solution for \eref{spde}  in this section, which do not
depend on any stopping times.

\begin{theorem}
Under the assumptions of Theorem 4.1, and suppose $u_0\in
L^p(\Omega,\mathbb{H}^0)$ for some $p\geq2$. Then for any $T>0$ we
have
\begin{eqnarray}{\mathbb{E}}\left(\sup_{0\leq t\leq T}\|u(t)\|_{L^2}^p+\int_0^T\|u(t)\|_{L^2}^{p-2}\|u(t)\|_{\theta_2}^2dt\right)
\leq C_T(1+{\mathbb{E}}\|u_0\|_{L^2}^p).\label{5.6a}
\end{eqnarray}
\end{theorem}

\begin{proof}
The result can be obtained by  using similar arguments as in the
proof of Lemma 4.1. \hspace{\fill}$\Box$
\end{proof}

\begin{theorem}
Suppose that $\theta_1\ge 0, \theta_2 > 0$ with $\theta_1+\theta_2 >
\frac{5}{4}$ and $u_0\in L^{mp}(\Omega,\mathbb{H}^0)\cap
L^p(\Omega,\mathbb{H}^1)$
 for some $p\geq2$  with a constant $m=m(\theta_2)>0$ given by
$$\left\{
                  \begin{array}{ll}
              m=1+\frac{1+\theta_2}{2(\theta_1+\theta_2-\frac{5}{4})}, & \text{if}~~\theta_1+\frac{\theta_2}{2}<\frac{5}{4}; \\
                   m=2+\frac{1}{\theta_2}, & \text{if}~~\theta_1+\frac{\theta_2}{2}\geq\frac{5}{4}.
                  \end{array}
                \right.
$$
If $g$ satisfies the Hypotheses (4.1) and (4.2), then for any $T>0$
we have
\begin{eqnarray}
{\mathbb{E}}\left(\sup_{0\leq t\leq T}\|u(t)\| ^p+\int_0^T\|u(t)\|
^{p-2}\|u(t)\|_{\theta_2+1}^2dt\right)\leq
C_{p,T}(1+\mathbb{E}\|u_0\|^p+C{\mathbb{E}}\|u_0\|_{L^2}^{mp}).\label{6.10}
\end
{eqnarray}
\end{theorem}
\begin{Rem} By the standard Krylov-Bogolyubov argument, an immediate
corollary of Theorem 4.4 is the existence of an invariant measure
for the associate transition semigroup. The topic of ergodicity for
this model will be investigated in a separate paper.

\end{Rem}
\begin{proof}
By Theorem 4.1, there exists a unique strong global solution for
\eref{spde}. Applying It\^{o}'s formula, we obtain that, for any
$T>0$, $t\in[0,T]$
\begin{eqnarray}
&&\|u(t)\|^p+p\nu\int_0^t\|u(s)\|^{p-2}\|u(s)\|_{\theta_2+1}^{2}ds\nonumber\\
=\!\!\!\!\!\!\!\!&&\|u_0\|^p
-p\int_0^t\|u(s)\|^{p-2}\langle\Lambda{B(u(s))},\Lambda{u}(s)\rangle
ds
\nonumber\\&&+\frac{p}{2}\int_0^t\|u(s)\|^{p-2}\|\Lambda{g(u(s))}\|_{L_2(U,\mathbb{H}^0)}^2ds\nonumber\\
&&+\frac{p(p-2)}{2}\int_0^t\|u(s)\|^{p-4}\|(\Lambda{g}(u(s)))^*\Lambda{u}(s)\|_{L^2}^{2}ds\nonumber\\
&&+p\int_0^t\|u(s)\|^{p-2}\langle\Lambda{g}(u(s))d{W}(s),\Lambda{u}(s)\rangle
.\label{6.11}
\end{eqnarray}
To bound the second term at the right hand side of \eref{6.11}, we
split it into two cases, namely, the case $\theta_2\leq1$ and the
case $\theta_2>1$. For the case $\theta_2\leq1$, thanks to the
commutator estimate, it directly yields
\begin{eqnarray}
|\langle\Lambda{B(u)},\Lambda{u}\rangle|\!\!\!\!\!\!\!\!&&=|\langle[\Lambda,Gu]\cdot\nabla
u^{n},\Lambda{u} \rangle|
\nonumber\\\!\!\!\!\!\!\!\!&&\leq{C}\|[\Lambda,Gu]\cdot\nabla
u\|_{L^{\frac{6}{3+2\theta_2}}}\|\Lambda{u}\|_{L^{\frac{6}{3-2\theta_2}}}
\nonumber\\\!\!\!\!\!\!\!\!&& \leq
C\left(\|\Lambda{G}u\|_{L^{p_1}}\|\Lambda
u\|_{L^{p_2}}+\|\Lambda{u}\|_{L^{p_3}}\|\Lambda{Gu}\|_{L^{p_4}}\right)\|\Lambda{u}\|_{L^{\frac{6}{3-2\theta_2}}}
.\nonumber
\end{eqnarray}
Here we choose
$$\frac{1}{p_1}=\frac{1}{p_4}=\frac{\delta_0+\theta_2}{3},~\frac{1}{p_2}=\frac{1}{p_3}=\frac{1}{2}-\frac{\delta_0}{3},$$
where
$\delta_0=\delta_0(\theta_2):=\{\frac{5}{2}-2\theta_1-\theta_2\}\vee0$.

 The assumptions $0<\theta_2\leq1$, $\theta_1\ge 0$ and $\theta_1+\theta_2 > \frac{5}{4}$
imply that $\delta_0\in[0,\theta_2)$,
${p_1},p_4\in(\frac{3}{2},\infty)$ and
 $p_2,p_3\in[2,6)$. By Lemmas 2.1 and 2.3, we have
$$\|\Lambda{Gu}\|_{L^{p_1}}\leq{C}\|Gu\|_{\frac{5}{2}-\theta_2-\delta_0}\leq{C}\|u\|_{L^2},$$
$$\|\Lambda{u}\|_{L^{p_2}}\leq C\|u\|_{\delta_0+1},$$
$$\|\Lambda{u}\|_{L^{\frac{6}{3-2\theta_2}}}\leq C\|u\|_{\theta_2+1}.$$
From the above estimates, we thus obtain
\begin{eqnarray}
|\langle\Lambda{B(u)},\Lambda{u}\rangle|
\leq{C}\|u\|_{L^2}\|u\|_{\delta_0+1}\|u\|_{\theta_2+1}.\nonumber
\end{eqnarray}

 For the case $\theta_2>1$, the Lemma 2.2 would suffice our purpose.
Actually, since $2\theta_1+\delta_0+\theta_2\geq\frac{5}{2}$,
applying Lemma 2.2 (with
$m_1=2\theta_1,m_2=\delta_0,m_3=\theta_2-1$), it leads to
\begin{eqnarray}
|\langle\Lambda{B(u)},\Lambda{u}\rangle|
\leq{C}\|u\|_{L^2}\|u\|_{\delta_0+1}\|u\|_{\theta_2+1}.\nonumber
\end{eqnarray}

 Therefore, by the Young's inequality and interpolation inequality, one may conclude that
\begin{eqnarray}
\left|p\|u\|^{p-2}\langle \Lambda{B(u)},\Lambda{u}\rangle\right|
\!\!\!\!\!\!\!\!\!&&\leq{C}\|u\|^{p-2}\|u\|_{L^2}\|u\|_{\delta_0+1}\|u\|_{\theta_2+1}
\nonumber\\\!\!\!\!\!\!\!\!&&\leq{C}\|u\|^{p-2}\|u\|_{L^2}^{\frac{2\theta_2+1-\delta_0}{\theta_2+1}}\|u\|_{\theta_2+1}^{\frac{\theta_2+2+\delta_0}{\theta_2+1}}
\nonumber\\\!\!\!\!\!\!\!\!&&\leq\varepsilon\|u\|^{p-2}\|u\|_{\theta_2+1}^{2}+C_{\varepsilon}\|u\|_{L^2}^{2m}\|u\|^{p-2}
\nonumber\\\!\!\!\!\!\!\!\!&&\leq\varepsilon\|u\|^{p-2}\|u\|_{\theta_2+1}^{2}+C\|u\|_{L^2}^{mp}+C\|u\|^{p}.\label{6.13}
\end{eqnarray}
By the H\"{o}lder's inequality and {Hypothesis} (4.2), we also get
\begin{eqnarray}
&&\frac{p}{2}\|u\|^{p-2}\|\Lambda{g}(u)\|_{L_2(U,\mathbb{H}^0)}^2+\frac{p(p-2)}{2}\|u\|^{p-4}\big\|(\Lambda{g}(u))^*\Lambda{u}\big\|_{L^2}^{2}\nonumber\\
\leq\!\!\!\!\!\!\!\!&&C\|u\|^{p-2}\|\Lambda{g}(u)\|_{L_2(U,\mathbb{H}^0)}^2\nonumber\\
\leq\!\!\!\!\!\!\!\!&&C(1+\|u\|^{p}).\label{6.14}
\end{eqnarray}
From \eref{6.11}-\eref{6.14}, we obtain
\begin{eqnarray}
&&\|u(t)\|^p+\int_0^t\|u^{n}(s)\|^{p-2}\|u(s)\|_{\theta_2+1}^{2}ds\nonumber\\
\leq\!\!\!\!\!\!\!\!&&\|u_0\|^p+CT+C\int_0^t\|u(s)\|^pds+C\int_0^t\|u(s)\|_{L^2}^{mp}ds
\nonumber\\
\!\!\!\!\!\!\!\!&&+\int_0^tp\|u(s)\|^{p-2}\langle\Lambda{g}(u(s))d{W}(s),\Lambda{u}(s)\rangle.\label{6.8}
\end{eqnarray}
For any $R_1>0$, define the stopping time
$$\tau_{R_1}:=\inf\{t\in[0,T]:\|u(t)\|>R_1\}\wedge{T}.$$
As a result of Theorem 4.1, we have
$$\lim_{R_1\rightarrow\infty}\tau_{R_1}=T,~~~~{\mathbb{P}}-a.s.$$
By the BDG's inequality we have
\begin{eqnarray}
&&{\mathbb{E}}\sup_{r\in[0,t\wedge\tau_{R_1}]}\left|\int_0^rp\|u(s)\|^{p-2}\langle\Lambda{g}(u(s))d{W}(s),{\Lambda}u(s)\rangle\right|
\nonumber\\
\leq\!\!\!\!\!\!\!\!&&3{\mathbb{E}}\left(\int_0^{t\wedge\tau_{R_1}}p\|u(s)\|^{2p-2}\|\Lambda{g}(u(s))\|_{L_2(U,\mathbb{H}^0)}^2ds
\right)^{1/2}\nonumber\\
\leq\!\!\!\!\!\!\!\!&&3{\mathbb{E}}\left(\sup_{s\in[0,t\wedge\tau_{R_1}]}\|u(s)\|^{2p-2}C\int_0^{t\wedge\tau_{R_1}}(1+\|u(s)\|^2)ds
\right)^{1/2}\nonumber\\
\leq\!\!\!\!\!\!\!\!&&3{\mathbb{E}}\left[\varepsilon\sup_{s\in[0,t\wedge\tau_{R_1}]}\|u(s)\|^p+C_{\varepsilon}
\left(\int_0^{t\wedge\tau_{R_1}}(1+\|u(s)\|^2)ds\right)^{p/2}\right]\nonumber\\
\leq\!\!\!\!\!\!\!\!&&3\varepsilon
{\mathbb{E}}\sup_{s\in[0,t\wedge\tau_{R_1}]}\|u(s)\|^p+3(2T)^{p/2-1}C_{\varepsilon}
{\mathbb{E}}\int_0^{t\wedge\tau_{R_1}}(1+\|u(s)\|^p)ds,\label{6.12}
\end{eqnarray}
where $\varepsilon>0$ is a small constant and $C_{\varepsilon}$
comes from the Young's inequality.

$\eref{5.6a},~ \eref{6.8}$ and \eref{6.12} yield that, for
$t\in[0,T]$,
\begin{eqnarray}
&&\mathbb{E}\sup_{0\leq t\leq\tau_{R_1}}\|u(t)\|^p+{\mathbb{E}}\int_0^{\tau_{R_1}}\|u(t)\|^{p-2}\|u(t)\|_{\theta_2+1}^{2}dt\nonumber\\
\leq\!\!\!\!\!\!\!\!&&{\mathbb{E}}\|u_0\|^p+C{\mathbb{E}}\|u_0\|_{L^2}^{mp}+CT+C{\mathbb{E}}
\int_0^{\tau_{R_1}}\|u(t)\|^pdt\nonumber\\
\leq\!\!\!\!\!\!\!\!&&
C(1+\mathbb{E}\|u_0\|^p+{\mathbb{E}}\|u_0\|_{L^2}^{mp})+C{\mathbb{E}}\int_0^{\tau_{R_1}}\|u(t)\|^pdt,\nonumber
\end{eqnarray}
where $C=C_{T,P}$ is a constant independent of $R_1$.

Thus, \eref{6.10} follows from the classic Gronwall's lemma and the
monotone convergence theorem.\hspace{\fill}$\Box$
\end{proof}
\section{Appendix}
\setcounter{equation}{0}
 \setcounter{definition}{0}

We recall the following stochastic Gronwall's lemma (cf. \cite[Lemma
5.3]{GZ1}) which is used in the proof of our results.
\begin{lemma}
Fix $T>0$. Assume that
$X,Y,Z,R:[0,T)\times\Omega\rightarrow\mathbb{R}$ are real-valued,
non-negative stochastic process. Let $\tau<T$ be a stopping time
such that $$\mathbb{E}\int_0^\tau(RX+Z)ds<\infty.$$ Assume,
moreover, that for some fixed constant $\kappa$,
$$\int_0^\tau{R}ds<\kappa,~~~\text{a.s.}$$
Suppose that for all stopping times $0\leq\tau_a<\tau_b\leq\tau$
\begin{eqnarray}{\mathbb{E}}\left(\sup_{t\in[\tau_a,\tau_b]}X+\int_{\tau_a}^{\tau_b}Yds\right)\leq c_0{\mathbb{E}}\left(X(\tau_a)+\int_{\tau_a}^{\tau_b}(RX+Z)ds\right),\label{a.1}
\end{eqnarray}
where $c_0$ is a constant independent of the choice of
$\tau_a,\tau_b$. Then
\begin{eqnarray}{\mathbb{E}}\left(\sup_{t\in[0,\tau]}X+\int_{0}^{\tau}Yds\right)\leq c{\mathbb{E}}\left(X(0)+\int_{0}^{\tau}Zds\right),\label{a.2}
\end{eqnarray}
where $c=c_{c_o,T,\kappa}$.
\end{lemma}

 \subsection*{Acknowledgements}
 The authors would like to thank Dr. Zhuan Ye for some helpful
 discussions.

\end{document}